\documentclass{siamart1116}
\usepackage{amssymb,amsmath,amsfonts,mathrsfs}
\usepackage{multirow}
\usepackage{etoolbox,verbatim,color}
\usepackage{url}
\newsiamthm{assumption}{{Assumption}}
\newsiamremark{remark}{{Remark}}
\newsiamthm{cor}{Corollary}
\newcommand{\iprod}[2]{\left\langle {#1}, {#2} \right\rangle}

\newcommand{\revise}[1]{\textcolor{black}{#1}} 

\usepackage[mathscr]{eucal}
\usepackage{epsfig,epsf,psfrag}
\usepackage{amssymb,amsfonts,latexsym}
\usepackage{amsmath}
\usepackage{graphicx}
\usepackage{epstopdf}
\usepackage{bm,xcolor,url}
\usepackage{array}
\usepackage{verbatim}
\usepackage{algpseudocode, cite}
\usepackage{algorithm}
\usepackage{verbatim}
\usepackage{textcomp}
\usepackage{mathrsfs}
\usepackage[referable]{threeparttablex}
\usepackage{subfig}
\usepackage{footnote}
\usepackage{enumitem}
\usepackage{xr}
\usepackage{mathtools}
\catcode`~=11 \def\UrlSpecials{\do\~{\kern -.15em\lower .7ex\hbox{~}\kern .04em}} \catcode`~=13 

\allowdisplaybreaks[3]

\newcommand{\tnorm}[1]{{\left\vert\kern-0.25ex\left\vert\kern-0.25ex\left\vert #1 
    \right\vert\kern-0.25ex\right\vert\kern-0.25ex\right\vert}}
\newcommand{\tnormt}[1]{{\vert\kern-0.25ex\vert\kern-0.25ex\vert #1 
    \vert\kern-0.25ex\vert\kern-0.25ex\vert}}

\newcommand{\norm}[1]{\left\Vert#1\right\Vert}
\newcommand{\normt}[1]{\Vert#1\Vert}
\newcommand{\abst}[1]{\vert#1\vert}
\newcommand{\abs}[1]{\left\lvert#1\right\rvert}

\newcommand{\nn}{\nonumber}

\newcommand{\defeq}{\triangleq}

\newcommand{\nt}{\addtocounter{equation}{1}\tag{\theequation}} 

\newcommand{\dom}{\mathsf{dom}\,}

\newcommand{\inter}{\mathsf{int}\,}
\newcommand{\bdry}{\mathsf{bd}\,}

\newcommand{\baromega}{\overline{\omega}}
\newcommand{\barlambda}{\overline{\lambda}}

\newcommand{\calA}{\mathcal{A}}
\newcommand{\calB}{\mathcal{B}}
\newcommand{\calC}{\mathcal{C}}

\newcommand{\calF}{\mathcal{F}}

\newcommand{\calK}{\mathcal{K}}

\newcommand{\calU}{\mathcal{U}}

\newcommand{\calX}{\mathcal{X}}
\newcommand{\calY}{\mathcal{Y}}

\newcommand{\rmc}{\mathrm{c}}

\newcommand{\rmd}{\mathrm{d}}
\newcommand{\rmD}{\mathrm{D}}

\newcommand{\rmP}{\mathrm{P}}

\newcommand{\bbE}{\mathbb{E}}

\newcommand{\bbI}{\mathbb{I}}

\newcommand{\bbN}{\mathbb{N}}

\newcommand{\bbP}{\mathbb{P}}

\newcommand{\bbR}{\mathbb{R}}

\newcommand{\bbU}{\mathbb{U}}

\newcommand{\bbZ}{\mathbb{Z}}

\newcommand{\barbbR}{\overline{\bbR}}

\newcommand{\scO}{\mathscr{O}}

\DeclareMathAlphabet{\mathbsf}{OT1}{cmss}{bx}{n}

\newcommand{\rvM}{\mathsf{M}}

\newcommand{\rvN}{\mathsf{N}}

\newcommand{\tilO}{\widetilde{O}}

\newcommand{\hatS}{\widehat{S}}

\newcommand{\tilS}{\widetilde{S}}

\newcommand{\tilu}{\widetilde{u}}

\newcommand{\tilx}{\widetilde{x}}

\newcommand{\baru}{\overline{u}}

\newcommand{\barx}{\overline{x}}

\newcommand{\barC}{\overline{C}}

\newcommand{\tlambda}{\widetilde{\lambda}}

\newcommand{\hpsi}{\widehat{\psi}}

\newcommand{\hlambda}{\widehat{\lambda}}

\newcommand{\lrangle}[2]{\left\langle{#1},{#2}\right\rangle}
\newcommand{\lranglet}[2]{\langle{#1},{#2}\rangle}

\newcommand{\eqa}{\stackrel{\rm(a)}{=}}
\newcommand{\eqb}{\stackrel{\rm(b)}{=}}
\newcommand{\eqc}{\stackrel{\rm(c)}{=}}

\newcommand{\lea}{\stackrel{\rm(a)}{\le}}
\newcommand{\leb}{\stackrel{\rm(b)}{\le}}
\newcommand{\lec}{\stackrel{\rm(c)}{\le}}

\newcommand{\gea}{\stackrel{\rm(a)}{\ge}}
\newcommand{\geb}{\stackrel{\rm(b)}{\ge}}

\newcommand{\ged}{\stackrel{\rm(d)}{\ge}}

\DeclareMathOperator*{\argmax}{arg\,max}
\DeclareMathOperator*{\argmin}{arg\,min}

\DeclareMathOperator{\st}{s.t.\;}

\newcommand{\qednew}{\nobreak \ifvmode \relax \else
      \ifdim\lastskip<1.5em \hskip-\lastskip
      \hskip1.5em plus0em minus0.5em \fi \nobreak
      \vrule height0.75em width0.5em depth0.25em\fi}

\newcommand{\TheTitle}
{An Inexact Primal-Dual Smoothing Framework for Large-Scale Non-Bilinear Saddle Point Problems}

\newcommand{\ShortTitle}{%
Inexact Primal Dual Smoothing for Large-scale Non-Bilinear SPP} 

\newcommand{\TheAuthors}{L. T. K. Hien, R. Zhao, W. B. Haskell}
\headers{\ShortTitle}{\TheAuthors}

\title{{\TheTitle}\thanks{
L.~T.~K.~Hien and R.~Zhao contribute equally to this work. R. Zhao’s research is supported by AFOSR Grant No. FA9550-22-1-0356.
}}
\author{
  Le Thi Khanh Hien\thanks{Department of Mathematics and Operations Research, University of Mons, Belgium
    (\email{thikhanhhien.le@umons.ac.be}).} 
 \and
  Renbo Zhao\thanks{Tippie College of Business, University of Iowa, USA (\email{renbo-zhao@uiowa.edu}).}
  \and
  William B. Haskell\thanks{Krannert School of Management, Purdue University, USA  (\email{whaskell@purdue.edu}).}
 }

\usepackage[referable]{threeparttablex}
\usepackage{footnote}
\usepackage{mathtools}
\usepackage{lipsum}

\allowdisplaybreaks
\numberwithin{theorem}{section}

\usepackage{amsopn}

\usepackage{amssymb,amsmath,amsfonts,mathrsfs}
\usepackage{etoolbox,verbatim,color}
\usepackage{url}
\usepackage{enumitem}

\ifpdf
\hypersetup{
  pdftitle={\TheTitle},
  pdfauthor={\TheAuthors}
}
\fi

\usepackage[referable]{threeparttablex}
\usepackage{footnote}
\usepackage{mathtools}
\usepackage{lipsum}

\allowdisplaybreaks

\begin{document}

\maketitle
\begin{abstract}
We develop an inexact primal-dual first-order smoothing framework to solve a class of non-bilinear saddle point problems with primal strong convexity. Compared with existing methods, our framework yields a significant improvement over the primal oracle complexity, while it has competitive dual oracle complexity. In addition, we consider the situation where the primal-dual coupling term has a large number of component functions. To efficiently handle this situation, we develop a randomized version of our smoothing framework, which allows the primal and dual sub-problems in each iteration to be inexactly solved  by randomized  algorithms in expectation. The convergence of this framework is analyzed both in expectation and with high probability. In terms of the primal and dual oracle complexities, this framework significantly improves over its deterministic counterpart.  
As an important application,  we adapt both frameworks for solving convex optimization problems with many functional constraints. To obtain an $\varepsilon$-optimal and $\varepsilon$-feasible solution, both frameworks achieve the best-known oracle complexities.
\end{abstract}

\setlength{\abovedisplayskip}{1.5ex}
\setlength{\belowdisplayskip}{1.5ex}

\section{Introduction.} 
\label{sec:intro}
We consider a class of convex-concave saddle point problems (SPPs) in which the primal-dual coupling function is not bilinear and has finite-sum structure, and the primal problem is strongly convex (see Section~\ref{sec:def} for details). This class of SPP has applications in numerous fields, including game theory, machine learning, and statistics (\revise{see Section \ref{sec:applications} for two specific examples}). In Section~\ref{sec:Applications_cvx_prog}, we  focus on one of the most important examples, {\em convex optimization problems with functional constraints}. For more applications, we refer readers to~\cite{Juditsky_12b,Chen_17,Hamed_18}  and the references therein.

 Over the past few years, bilinear SPPs have been studied extensively, with many efficient algorithms proposed, e.g., Nesterov smoothing~\cite{Nest_05} and the primal-dual hybrid gradient method (PDHG)~\cite{Chambolle_16}. For non-bilinear non-smooth SPPs without composite structure, the existing algorithms are mainly based on the primal-dual subgradient method (e.g.,~\cite{Nedic_09,Nest_09}). For non-smooth SPPs with composite structure (e.g., the problem in~\eqref{eq:SP}, see Section~\ref{sec:def}), the algorithms for smooth SPPs can be extended and their oracle complexities can be improved. For example,  the Mirror-Prox method~\cite{Nemi_05} was extended in~\cite{Juditsky_12b} to solve~\eqref{eq:SP}, Nesterov smoothing was extended in~\cite{Kolo_17} to solve a class of convex-linear SPPs. Recently, the PDHG method, originally developed for bilinear SPPs, was extended in~\cite{Hamed_18,Zhao_19} to solve~\eqref{eq:SP}.
In addition, the stochastic and randomized versions of the aforementioned methods have also been developed in~\cite{Juditsky_11,Hamed_18b,Zhao_19} to tackle the situation where only stochastic first-order oracles are available.  However, most of these methods focus on the case where the primal problem is non-strongly convex, so the favorable strong convexity condition may not be exploited to further improve the oracle complexity. Two works that have utilized this condition are~\cite{Juditsky_12b,Hamed_18}, which are based on the Mirror-Prox and PDHG methods, respectively. However, both works only consider the special case of a strongly convex - linear SPP.  Therefore, in this work, we aim to answer the following question: \emph{ Can we develop an algorithmic framework that works for strongly convex - generally concave SPP, yet with improved oracle complexity over the existing methods\footnote{Note that one year after our paper was published online (see arXiv:1711.03669v2), the work \cite{Theku_19} also addressed this question \revise{under the Euclidean setting}, 
with the same oracle complexity as our method.}?}


We make three main contributions.  
First, we develop a (deterministic) inexact primal-dual smoothing (IPDS) framework (i.e., Algorithm~\ref{alg:SMA}) for solving the non-bilinear SPP in~\eqref{eq:SP} with primal strong convexity. Compared with existing works for the case where the dual problem is generally concave, the primal oracle complexity of our framework is significantly better, while the dual oracle complexity is competitive. In addition, in contrast  to the methods in~\cite{Juditsky_12b,Hamed_18}, which can {\em only} improve the primal oracle complexity for strongly convex-linear SPP, our framework applies to strongly convex-generally concave SPP. 
Second, we develop a randomized version of our IPDS framework (i.e., Algorithm~\ref{alg:SRSA}), by allowing each sub-problem to be solved inexactly {\em in expectation}. This framework is particularly useful in the regime where $n$ is large in~\eqref{eq:finite-sum}.  
 In addition, we show that  Algorithm~\ref{alg:SRSA} converges {\em with high probability}. \revise{See  the discussion at the end of Section \ref{sec:rand_framework} for a summary of the complexity results and see Table~\ref{table:Primal_Dual_norm} for a comparison of the complexities of our proposed method with existing methods.}
Finally, we apply both of our aforementioned frameworks (i.e., Algorithms~\ref{alg:SMA} and~\ref{alg:SRSA}) to convex optimization problems with (potentially many) functional constraints (more precisely, to their associated Lagrangian problems; see Section~\ref{sec:cvx_problem} for details). To do so, we manage to overcome two challenges: the unboundedness of the dual feasible region and the dependence of the smoothness parameter on the dual variable (see Section~\ref{sec:oracle_complexity_det}). To obtain an $\varepsilon$-optimal  and $\varepsilon$-feasible solution (cf.~\eqref{eq:primal_criterion}), both Algorithms~\ref{alg:SMA} and~\ref{alg:SRSA} achieve the state-of-the-art (primal) oracle complexities. Compared to other first-order methods with similar oracle complexities but {\em specifically} designed for constrained convex optimization problems, our frameworks enjoy much wider applicability. 

\revise{The rest of the paper is organized as follows. In the next section, we provide the problem setting and some preliminaries that are used in the sequel. In Section \ref{sec:det_framework}, we develop the deterministic IPDS framework  and present its convergence analysis and oracle complexity. In Section \ref{sec:rand_framework}, we develop a randomized version of the  IPDS framework, we establish its oracle complexity, and then we present its convergence with high probability. Finally,  we apply  our IPDS frameworks to solve  the saddle point problems associated with the constrained convex optimization problems in Section \ref{sec:Applications_cvx_prog}.   }

\section{Problem setting and preliminaries.} 
\label{sec:def}

\paragraph{Notations.} We denote $\barbbR\defeq \bbR\cup\{+\infty\}$, $\bbR_+\defeq[0,+\infty)$, $\bbZ_+\defeq \bbN\cup\{0\}$, and $[n]\defeq \{1,\ldots,n\}$. For a function $g:\bbE\to\barbbR$, we denote its domain by $\dom g\defeq \{x\in\bbE:g(x)<+\infty\}$. For any set $\calK$, $\inter\calK$ and $\bdry\calK$ denote the interior and boundary of $\calK$, respectively.

\revise{In the following, we formally define the class of convex-concave saddle point problems that are studied in this paper.  Two specific SPP applications are given in Section \ref{sec:applications}. In Section \ref{sec:oracle_model} and Section \ref{sec:def_1}, we define the primal and dual first-order oracles, the primal and dual functions, the duality gap associated with the saddle function, their smoothed versions, and the inexact solutions of the optimization problems appearing in these definitions. 
We then prove some smoothness properties of the formerly defined quantities in Section \ref{sec:smoothproperties}.  These smoothness properties are involved in the convergence analysis and the oracle complexity of our upcoming algorithms. More specifically, Proposition \ref{lem:GradLip} is used to establish the reduction of the smoothed duality gap in Lemma \ref{lem:SmDuGap}. It is also used to establish that the oracle complexities in Section \ref{sec:oracle_comp_det} and Section \ref{sec:oracle_comp_stoc} depend on the constant $L_D$ given in Proposition \ref{prop:Lsmooth}.  }

\subsection{Problem setting}
Let $(\mathbb E_{1},\norm{\cdot}_{\mathbb E_{1}})$ and $(\mathbb E_{2},\norm{\cdot}_{\mathbb E_{2}})$  be finite-dimensional real normed spaces, with dual spaces $(\mathbb E^*_{1},\norm{\cdot}_{\mathbb E^*_{1}})$ and $(\mathbb E^*_{2},\norm{\cdot}_{\mathbb E^*_{2}})$, respectively. We consider the following  SPP

 \begin{equation}
{\min}_{x\in \calX}\;{\max}_{\lambda\in\Lambda}\;\big\{S(x,\lambda)\defeq f(x) + g(x) + \Phi(x,\lambda) - h(\lambda)\big\},\label{eq:SP}
\end{equation}
where $\calX\subseteq \bbE_1$ and $\Lambda\subseteq\bbE_2$ are nonempty, convex, and closed sets, $\calX$ is bounded and the functions $f,g:\bbE_1\to\barbbR$ and $h:\bbE_2\to\barbbR$ are convex, closed, and proper (CCP) functions.  We assume that $\calX\subseteq\dom g$,  $\Lambda\subseteq\dom h$, and both $g$ and $h$ admit tractable Bregman proximal projections (BPP) on $\calX$ and $\Lambda$, respectively (see Section~\ref{sec:def_1} for its precise definition). 
In addition, we assume that $f$ is differentiable on an open set $\calX'\supseteq\calX$, and $\mu$-strongly convex (s.c.) and $L$-smooth on $\calX$ (where $L\ge \mu > 0$),

 \begin{equation}\label{eq:sm_sc_f}
\frac{\mu}{2} \norm{x-x'}_{\bbE_1}^2 \le f(x) - f(x') - \lrangle{\nabla f(x')}{x-x'}\le \frac{L}{2}\norm{x-x'}_{\bbE_1}^2,\quad\forall\,x,x'\in \calX, 
\end{equation}

where $\lranglet{\cdot}{\cdot}$ denotes the duality pairing between $\bbE_1^*$  (resp.\ $\bbE_2^*$) and $\bbE_1$  (resp.\ $\bbE_2$).
\revise{It is worth noting that if $f$ is twice continuously differentiable then the $L$-smoothness condition over the closed and bounded set $\mathcal X$ is automatically satisfied.} 

We next state our assumptions on the function $\Phi:\bbE_1\times\bbE_2\to\bbR$.  First, it is convex-concave, i.e., for any $(x,\lambda)\in\bbE_1\times\bbE_2$, $\Phi(\cdot,\lambda)$ is convex  on $\bbE_1$ and  $\Phi(x,\cdot)$ is concave on $\bbE_2$. In addition, $\Phi$ satisfies the $(L_{xx},L_{\lambda x},L_{\lambda \lambda})$-smoothness condition (where $L_{xx},L_{\lambda x},L_{\lambda \lambda}\ge 0$) on $\calX\times\Lambda$, i.e., for any $x,x'\in\calX$ and $\lambda,\lambda'\in\Lambda$, 

\begin{subequations}
\begin{align}
&\norm{\nabla_x \Phi(x,\lambda)-\nabla_x \Phi(x',\lambda)}_{\bbE_1^*}\le L_{xx}\norm{x-x'}_{\bbE_1},\label{eq:Phi_sm(a)} \\
&\norm{\nabla_x \Phi(x,\lambda)-\nabla_x \Phi(x,\lambda')}_{\bbE_1^*}\le L_{\lambda x}\norm{\lambda-\lambda'}_{\bbE_2},\label{eq:Phi_sm(b)}\\
&\norm{\nabla_\lambda \Phi(x,\lambda)-\nabla_\lambda \Phi(x',\lambda)}_{\bbE_2^*}\le L_{\lambda x}\norm{x-x'}_{\bbE_1},\label{eq:Phi_sm(c)}\\
&\norm{\nabla_\lambda \Phi(x,\lambda)-\nabla_\lambda \Phi(x,\lambda')}_{\bbE_2^*}\le L_{\lambda\lambda}\norm{\lambda-\lambda'}_{\bbE_2},\label{eq:Phi_sm(d)}
\end{align}
\end{subequations}

where $x\mapsto\nabla_x \Phi(x,\lambda)$ and $\lambda\mapsto\nabla_\lambda \Phi(x,\lambda)$ denote the gradients of $\Phi(\cdot,\lambda)$ and $\Phi(x,\cdot)$, respectively. For later use, let us define the (primal) condition number
\begin{equation*}
\kappa_\calX \defeq (L+L_{xx})/\mu. 
\end{equation*}
We assume that the saddle function $\Phi(\cdot,\cdot)$ has the following finite-sum structure
\begin{equation}
\Phi(x,\lambda) \defeq (1/n) \textstyle{\sum}_{i=1}^n \Phi_i(x,\lambda),\label{eq:finite-sum}
\end{equation}
where for each $i\in[n]$ and any $(x,y)\in\bbE_1\times\bbE_2$, $\Phi_i:\bbE_1\times\bbE_2\to[-\infty,+\infty]$ is convex-concave and satisfies the $(L^i_{xx},L^i_{\lambda x},L^i_{\lambda \lambda})$-smoothness condition on $\calX\times\Lambda$. As a result, the smoothness parameters of $\Phi$ can be bounded by $L_{xx}\le (1/n)\sum_{i=1}^nL^i_{xx}$, $L_{\lambda x}\le (1/n)\sum_{i=1}^nL^i_{\lambda x}$ and $L_{\lambda \lambda}\le (1/n)\sum_{i=1}^nL^i_{\lambda \lambda}$. In addition, we are particularly interested in the setting where the number of component functions (i.e., $n$) is {\em large}. 

For well-posedness, we assume that for the SPP in~\eqref{eq:SP}, at least one saddle point $(x^*,\lambda^*)$ exists, i.e., there exists $(x^*,\lambda^*)\in\calX\times\Lambda$ such that
 
\begin{equation}
S(x^*,\lambda)\le S(x^*,\lambda^*)\le S(x,\lambda^*), \quad\forall\,(x,\lambda)\in\calX\times\Lambda.\label{eq:def_saddle_point}
\end{equation}


\revise{
\subsection{Applications.}\label{sec:applications}
In the following, we illustrate  two specific applications of the SPP in~\eqref{eq:SP}. In Section~\ref{sec:Applications_cvx_prog}, we introduce another important application,  namely convex optimization with functional constraints. 
More applications can be found in~\cite{Juditsky_12b,Chen_17,Hamed_18}  and the references therein. 
}

\revise{
{\em Application I: Two-player zero-sum game with non-linear payoff}~\cite{Chen_17}. 
Define the following sets:
\begin{equation}
\calX\defeq \{x\geq 0:\textstyle\sum_{i=1}^n x_i=N\} \quad \mbox{and}\quad  \calY\defeq \{y\geq 0:\textstyle\sum_{i=1}^n y_i=P\},
\end{equation}
where $N,P>0$ are problem constants. Let us consider the following problem:   
\begin{equation}
\label{eq:game2player}
\min_{x\in\calX} \max_{y\in\calY} \;\;\frac{\alpha}{2}\|x\|_2^2 + \sum_{i=1}^{n}\log\left(1+\frac{\beta_{i}y_{i}}{\sigma_{i}+x_{i}}\right)-\frac{\gamma}{2}\langle Qy,y\rangle,
\end{equation}
where $\beta_i,\sigma_i>0$, for all $i\in [n]$, $Q\succeq 0$ (i.e., $Q$ is a symmetric and positive semidefinite matrix), and $\alpha,\gamma\ge 0$. 
Indeed, Problem~\eqref{eq:game2player} is motivated by the classical water-filling problem in information theory (see, e.g.,~\cite{cover2006}). Specifically, $n$ denotes the number of Gaussian communication channels, and each channel $i$ has noise power $\sigma_i$, for all $i\in[n]$. Let Alice and Bob be two players, with noise power $N$ and transmission power $P$, respectively. Alice wishes to allocate her noise power  to these $n$ channels to minimize the total channel capacity, whereas Bob aims to allocate his transmission power  to these $n$ channels to maximize the total channel capacity. 
Let $x = (x_1,\ldots,x_n)$ and $y = (y_1,\ldots,y_n)$ denote the allocation plans of Alice and Bob, respectively. The total channel capacity is given by $\sum_{i=1}^{n}\log\left(1+{\beta_{i}y_{i}}/({\sigma_{i}+x_{i}})\right)$, where $\{\beta_i\}_{i\in[n]}$ are normalization constants. Apart from this, the allocation plans $x$ and $y$ incur costs $({\alpha}/{2})\|x\|_2^2$ for Alice and $({\gamma}/{2})\langle Qy,y\rangle$ for Bob, respectively. Therefore, Problem~\eqref{eq:game2player} reflects the zero-sum game where Alice (resp.\ Bob) wishes to minimize (resp.\ maximize) the total channel capacity, while accounting for her (resp.\ his) allocation cost. The saddle points of Problem~\eqref{eq:game2player} are precisely the Nash equilibria of this game. 
}

\revise{
{\em Application II: Proximal sub-problem of nonconvex-concave optimization}~\cite{Theku_19,Zhao_20}. 
Let $\bbE_1$ be a finite-dimensional real Hilbert space. 
Consider the SPP in~\eqref{eq:SP},  where $f\equiv 0$, $\Phi(\cdot,\lambda)$ is $\gamma$-weakly convex for any $\lambda\in\Lambda$ ($\gamma\ge 0$), i.e.,
\begin{equation}
-(\gamma/2)\norm{x'-x}^2_{\bbE_1} \le \Phi(x',\lambda) - \Phi(x,\lambda) - \lranglet{\nabla_x \Phi(x,\lambda)}{x'-x}, \label{eq:weak_cvx}
\end{equation}
and all other assumptions are unchanged. In this setting, we wish to find a near-stationary point of the {\em primal function} $x\mapsto\max_{\lambda\in\Lambda} S(x,\lambda)$ on $\calX$ (see \cite{Theku_19} for its definition). To do so, a natural strategy is to apply the proximal point method to the primal function. Consequently, at each iteration of this method, given the current iterate $x\in\calX$, one needs to solve the following sub-problem:
\begin{equation}
{\min}_{x'\in \calX}\;{\max}_{\lambda\in\Lambda}\;\big\{\tilS(x',\lambda)\defeq g(x') + \gamma\norm{x'-x}^2_{\bbE_1} + \Phi(x',\lambda) - h(\lambda)\big\}. \label{eq:new_SP}
\end{equation}
It is clear that $\tilS(\cdot,\lambda)$ is $(\gamma/2)$-strongly convex on $\calX$ for any $\lambda\in\Lambda$. Therefore, Problem~\eqref{eq:new_SP} is indeed an instance of Problem~\eqref{eq:SP}. 
}

\subsection{Primal and Dual First-order Oracles.}\label{sec:oracle_model}

Since we are interested in developing primal-dual first-order methods to solve the SPP in~\eqref{eq:SP}, where $\Phi(\cdot,\cdot)$ has the finite-sum structure as in~\eqref{eq:finite-sum}, we  set up the primal and dual first-order oracles as follows: Upon receiving $(x,\lambda,i)\in\calX\times\Lambda\times[n]$, %
the primal oracle $\scO^\rmP$ returns 
$\nabla_x \Phi_i(x,\lambda)$  and the dual oracle $\scO^\rmD$ returns $\nabla_\lambda \Phi_i(x,\lambda)$. In addition, $\scO^\rmP$ returns $\nabla f(x)$ upon receiving $(x,0)$. Accordingly, we define the {\em primal oracle complexity}  and {\em dual oracle complexity} to be the number of oracle calls of $\scO^\rmP$ and $\scO^\rmD$,  respectively.  

\subsection{Definitions.}\label{sec:def_1}

First, we define the primal function $\psi^\rmP:\bbE_1\to\barbbR$ and the dual function $\psi^\rmD:\bbE_2\to\barbbR$ associated with $S(\cdot,\cdot)$ as 

\begin{align}
\psi^\rmP(x)&\triangleq\sup\nolimits_{\lambda\in\Lambda} {S}(x,\lambda)= f(x) + g(x) + \hpsi^\rmP(x),\;\quad\forall\,x\in \bbE_1,\label{eq:primalF}\\
\psi^\rmD(\lambda)&\triangleq\inf\nolimits_{x\in \calX}S(x,\lambda) = \hpsi^\rmD(\lambda) - h(\lambda), \;\;\quad\quad\quad \quad\forall\,\lambda\in\bbE_2, \label{eq:dualF}
\end{align} 
where
\begin{align}
\hpsi^\rmP(x)&\defeq \sup\nolimits_{\lambda\in\Lambda} \{\hatS^\rmD(x,\lambda)\defeq\Phi(x,\lambda) - h(\lambda)\}, \quad\,\quad\quad\quad \forall\,x\in \bbE_1,\label{eq:hprimalF}\\
\hpsi^\rmD(\lambda)&\defeq\inf\nolimits_{x\in \calX} \{\hatS^\rmP(x,\lambda)\defeq f(x) + g(x) + \Phi(x,\lambda)\}, \;\quad\forall\,\lambda\in\bbE_2. \label{eq:hdualF}
\end{align}

Let $\omega:\bbE_2\to\barbbR$ be a CCP function that is 1-s.c.\ and continuous on $\Lambda$ and essentially smooth, i.e., $\omega$ is continuously differentiable on $\inter\dom \omega\ne \emptyset$, and $\normt{\nabla w(\lambda_k)}_*\to+\infty$ as $\lambda_k\to\lambda\in\bdry\dom\omega$~\cite{Bauschke_01}.  In addition, for any $\alpha>0$ and $\upsilon\in\bbE_2^*$, the  following problem (which is defined by the triple $(\omega,h,\Lambda)$)

\begin{equation}
{\min}_{\lambda\in\Lambda}\; h(\lambda) + \lranglet{\upsilon}{\lambda} + \alpha^{-1} \omega(\lambda)\label{eq:BPP}
\end{equation} 
has a (unique) {\em easily computable} solution in 
$\Lambda^o\defeq\Lambda\cap\inter\dom \omega$. 
We call $\omega$ a {\em distance generating function} (DGF) w.r.t.\ $(h,\Lambda)$. Additionally, we say that $h$ has a tractable BPP on $\Lambda$ if and only if such an $\omega$ exists. Since we also assume that $g$ has a tractable BPP on $\calX$ (cf.\ Section~\ref{sec:intro}), there exists a DGF $\baromega:\bbE_1\!\to\!\barbbR$ w.r.t.\ $(g,\calX)$. Similar to $\Lambda^o$, we also define $\calX^o\defeq \calX\cap\inter\dom \baromega$. The assumption  that~\eqref{eq:BPP} can be solved easily is typical in the literature on first-order methods for composite optimization on normed spaces (see e.g.,~\cite{Teboulle_18}). Since  we  employ such methods to solve the sub-problems in our framework (cf.\ Sections~\ref{sec:solve_subproblem_det} and~\ref{sec:solve_subproblem_stoc}), we also make this assumption throughout the whole work. 
 
Based on $\omega$, we define  the dual-regularized saddle function 

\begin{equation}
S_{\rho}(x,\lambda)\triangleq S(x,\lambda)-\rho\,\omega\left(\lambda\right),\label{eq:regularized_saddle_function}
\end{equation}
where $\rho>0$ is the smoothing parameter. \revise{By introducing the regularization for the dual saddle function, we are able to define the $\rho$-smoothed duality gap later in \eqref{eq:smoothed_duality_gap}. The $\rho$-smoothed duality gap plays an important role in our upcoming convergence analysis.} The primal function $\psi_\rho^\rmP:\bbE_1\to\barbbR$ associated with $S_\rho(\cdot,\cdot)$ is 

\begin{align}
\psi_\rho^\rmP(x)&\triangleq\sup\nolimits_{\lambda\in\Lambda} {S}_{\rho}(x,\lambda)= f(x) + g(x) + \hpsi_\rho^\rmP(x),\;\quad\forall\,x\in \bbE_1,\label{eq:primalF_sm}
\end{align}
where
\begin{equation}
\hpsi_\rho^\rmP(x) \defeq \sup\nolimits_{\lambda\in\Lambda}\{\hatS^\rmD_\rho(x,\lambda)\defeq \Phi(x,\lambda) - h(\lambda) - \rho\omega(\lambda)\}, \;\quad\forall\,x\in \bbE_1. \label{eq:hprimalF_sm}
\end{equation}
Next, we introduce the optimal solutions of the optimization problems in~\eqref{eq:hdualF} and~\eqref{eq:hprimalF_sm}. 
Since $f$ is $\mu$-s.c.\ on $\calX$, the minimization problem in~\eqref{eq:dualF} has the unique solution
\begin{equation}
x^*(\lambda) \defeq {\argmin}_{x\in\calX} \hatS^\rmP(x,\lambda),\quad\;\forall\,\lambda\in\bbE_2. \label{eq:primal_opt_sln}
\end{equation}
In addition, for any $\lambda\in\bbE_2$ and $\gamma\ge 0$, we call $\tilx_\gamma(\lambda)\in\calX$ an $\gamma$-inexact solution if 
\begin{equation}
\label{eq:inexact_primal_sln}
\hatS^\rmP(\tilx_\gamma(\lambda),\lambda) - \hpsi^\rmD(\lambda)=\hatS^\rmP(\tilx_\gamma(\lambda),\lambda) - \hatS^\rmP(x^*(\lambda),\lambda) \le \gamma. 
\end{equation}
Similar to~\eqref{eq:hdualF}, since $\hatS^\rmD_\rho(x,\cdot)$ is $\rho$-strongly concave on $\Lambda$, the maximization problem in~\eqref{eq:hprimalF_sm} has the unique solution

\begin{align}
\lambda^*_{\rho}(x) \defeq {\argmax}_{\lambda\in\Lambda} \hatS^\rmD_\rho(x,\lambda),\quad\;\forall\,x\in\bbE_1. 
\end{align}
For any $x\in\bbE_1$ and $\eta\ge 0$, we call $\tlambda_{\rho,\eta}(x)\in\Lambda$ an $\eta$-inexact solution if

\begin{equation}
\hpsi^\rmP_\rho(\lambda) - \hatS^\rmD_\rho(x,\tlambda_{\rho,\eta}(x)) =\hatS^\rmD_\rho(x,\lambda^*_{\rho}(x)) - \hatS^\rmD_\rho(x,\tlambda_{\rho,\eta}(x))  \le \eta. \label{eq:inexact_dual_sln}
\end{equation}
We assume that inexact solutions of~\eqref{eq:hdualF} and~\eqref{eq:hprimalF_sm} can be found efficiently. 
 Note that optimization algorithms admitting inexactness in solving the auxiliary subproblems have been studied extensively, see, e.g.,~\cite{Doikov_19,Doikov_20,Hien2019,Salvo2012,Schmidt_11} and the references therein.   

Finally, we define the duality gap $\Delta:\bbE_1\times \bbE_2\to\barbbR$  associated with $S(\cdot,\cdot)$ as 

\begin{equation}
\Delta(x,\lambda)\triangleq\psi^\rmP(x)-\psi^\rmD(\lambda), \quad\;\forall\,(x,\lambda)\in \bbE_1\times\bbE_2. 
\label{eq:dualityGap}
\end{equation}
Clearly, $(x^*,\lambda^*)\in\calX\times\Lambda$ is a saddle point of~\eqref{eq:SP} if and only if $\Delta(x^*,\lambda^*)=0$.  
Then, for any $\varepsilon>0$, we call $(\bar{x},\bar{\lambda})\in \calX\times\Lambda$ an $\varepsilon$-saddle point of~\eqref{eq:SP} if $\Delta(\bar{x},\bar{\lambda})\le \varepsilon$. 
Additionally, we define the ($\rho$-)smoothed duality gap $\Delta_\rho:\bbE_1\times\bbE_2\to\barbbR$ as 

\begin{equation}
\Delta_{\rho}(x,\lambda)\defeq\psi_{\rho}^{\rmP}(x)-\psi^\rmD(\lambda),\quad\;\forall\,(x,\lambda)\in \bbE_1\times\bbE_2.\label{eq:smoothed_duality_gap}
\end{equation}

\subsection{Smoothness Properties.}
\label{sec:smoothproperties}
We first show that the optimal solution $x^*(\cdot)$ in~\eqref{eq:primal_opt_sln} is Lipschitz on $\bbE_2$ and that the function $\hpsi^\rmD$ in~\eqref{eq:hdualF} is smooth on $\bbE_2$. 


\begin{proposition}
\label{prop:Lsmooth}
The function $\hpsi^\rmD$ is differentiable on $\bbE_2$ and for any $\lambda\in\bbE_2$, $\nabla\hpsi^\rmD(\lambda)=\nabla_{\lambda}\hatS^\rmP(x^*(\lambda),\lambda)=\nabla_{\lambda}\Phi(x^*(\lambda),\lambda)$. In addition, $x^*(\cdot)$ is $({L_{\lambda x}}/{\mu})$-Lipschitz on $\bbE_2$ and $\nabla\hpsi^\rmD$ is $L_\rmD$-Lipschitz on $\bbE_2$, where 

\begin{equation}
L_\rmD\triangleq L_{\lambda\lambda}+{L_{\lambda x}^{2}}/\mu. \label{eq:L_D}
\end{equation} 
\end{proposition}

\proof{}
Since for any $(x,\lambda)\in\bbE_1\times\bbE_2$, $\calX$ is compact, $\hatS^\rmP(\cdot,\lambda)$ is $\mu$-s.c.\ on $\calX$,  and $\hatS^\rmP(x,\cdot)$ is differentiable on $\bbE_2$, we use Danskin's Theorem~\cite[Proposition~B.25]{Bert_99}  
to conclude that $\hpsi^\rmD$ is differentiable on $\bbE_2$ and $\nabla \hpsi^\rmD(\lambda)=\nabla_\lambda \hatS^\rmP(x^*(\lambda),\lambda)$. To show Lipschitz continuity of $x^*(\cdot)$, note that for any $\lambda_1,\lambda_2\in\bbE_2$, since $\hatS^\rmP(\cdot,\lambda_1)$ is $\mu$-s.c., 

\begin{equation}
\label{eq:StorngConvexity}
\begin{split}
\|x^*(\lambda_2)-x^*(\lambda_1)\|^2&\le ({2}/{\mu}) \big(\hatS^\rmP(x^*(\lambda_2),\lambda_1)-\hatS^\rmP(x^*(\lambda_1),\lambda_1)\big)\\
\|x^*(\lambda_2)-x^*(\lambda_1)\|^2&\le ({2}/{\mu}) \big(\hatS^\rmP(x^*(\lambda_1),\lambda_2)-\hatS^\rmP(x^*(\lambda_2),\lambda_2)\big).
\end{split} 
\end{equation}
On the other hand, 

\begin{equation}
\label{eq:Lips_S^rmP}
\begin{array}{ll}
&\big(\hatS^\rmP(x^*(\lambda_2),\lambda_1)-\hatS^\rmP(x^*(\lambda_1),\lambda_1)\big) - \big(\hatS^\rmP(x^*(\lambda_2),\lambda_2)-\hatS^\rmP(x^*(\lambda_1),\lambda_2)\big)\\
\eqa & \big(\Phi(x^*(\lambda_2),\lambda_1)-\Phi(x^*(\lambda_1),\lambda_1)\big) - \big(\Phi(x^*(\lambda_2),\lambda_2)-\Phi(x^*(\lambda_1),\lambda_2)\big)\\
= & \int_0^1 \big\langle\nabla_\lambda \Phi(x^*(\lambda_2),\lambda_2+t(\lambda_1-\lambda_2))-\nabla_\lambda \Phi(x^*(\lambda_1),\lambda_2+t(\lambda_1-\lambda_2)),\lambda_1-\lambda_2\big\rangle \rmd t \\
\leb & \normt{\nabla_\lambda \Phi(x^*(\lambda_2),\lambda_2+t(\lambda_1-\lambda_2))-\nabla_\lambda \Phi(x^*(\lambda_1),\lambda_2+t(\lambda_1-\lambda_2))}_*\normt{\lambda_1-\lambda_2}\\
\lec & L_{\lambda x} \normt{x^*(\lambda_2)-x^*(\lambda_1)}\normt{\lambda_1-\lambda_2}, 
\end{array}
\end{equation} 
where in (a) we use the definition of $\hatS^\rmP(\cdot,\cdot)$ in~\eqref{eq:primal_opt_sln}, in (b) we use the definition of the dual norm $\normt{\cdot}_*$, and in (c) we use the Lipschitz continuity of $\nabla_\lambda \Phi(\cdot,\lambda)$ in~\eqref{eq:Phi_sm(c)}. 
Therefore, by combining~\eqref{eq:StorngConvexity} and~\eqref{eq:Lips_S^rmP}, 
we have 

\begin{equation}
\| x^*(\lambda_1)-x^*(\lambda_2)\|\leq \frac{L_{\lambda x}}{\mu}\|\lambda_1-\lambda_2\|.\label{eq:bound_x_via_lambda}
\end{equation}
As a result, 

\begin{equation*}
\begin{array}{ll}
&\normt{\nabla \hpsi^\rmD(\lambda_1)-\nabla \hpsi^\rmD(\lambda_2)}_{*}=\left\| \nabla_\lambda  \Phi(x^*(\lambda_1),\lambda_1)-\nabla_\lambda  \Phi(x^*(\lambda_2),\lambda_2)\right\|_{*}\\
&\le \quad \norm{\nabla_\lambda  \Phi(x^*(\lambda_1),\lambda_1)-\nabla_\lambda  \Phi(x^*(\lambda_2),\lambda_1)}_{*} + \norm{\nabla_\lambda  \Phi(x^*(\lambda_2),\lambda_1)-\nabla_\lambda  \Phi(x^*(\lambda_2),\lambda_2)}_{*}\\
 &\le\quad L_{\lambda x} \| x^*(\lambda_1)-x^*(\lambda_2)\| + L_{\lambda\lambda}\|\lambda_1-\lambda_2\|\leq \quad\left( L_{\lambda x}^2/\mu+L_{\lambda\lambda}\right)\|\lambda_1-\lambda_2\|,
\end{array}
\end{equation*}

where in the last inequality we use~\eqref{eq:bound_x_via_lambda}. 
~\hfill\qed\endproof

By a symmetric argument, we can also conclude that $\hpsi^\rmP_\rho$ is differentiable on $\bbE_1$ and $\nabla \hpsi^\rmP_\rho$ is $(L_{xx}+L_{\lambda x}^2/\rho)$-Lipschitz on $\bbE_1$. For this reason, we  call $\hpsi^\rmP_\rho$ the ($\rho$-){\em smoothed primal function} and consequently, $\Delta_{\rho}$ the ($\rho$-){\em smoothed duality gap}. 

Based on Proposition~\ref{prop:Lsmooth}, we can establish the following results involving $\tilx_{\gamma}(\lambda)$, i.e., the $\gamma$-inexact solution of~\eqref{eq:dualF}. 

\begin{proposition} \label{lem:GradLip}
For any $\gamma\ge 0$ and $\lambda,\lambda'\in\bbE_2$, we have 

\begin{align}
&\|\nabla_{\lambda}\hatS^\rmP(\tilx_{\gamma}(\lambda),\lambda)-\nabla\hpsi^\rmD(\lambda)\|_{*} 
\le L_{\lambda x}\sqrt{2\gamma/\mu},\label{eq:lips_grad} \\
&0\leq \hatS^\rmP(\tilx_{\gamma}(\lambda),\lambda)-\hpsi^\rmD(\lambda')+\big\langle \nabla_{\lambda}\hatS^\rmP(\tilx_{\gamma}(\lambda),\lambda),\lambda'-\lambda\big\rangle \leq L_\rmD\|\lambda-\lambda'\|^2+2\gamma. 
\label{eq:DescentLem}
\end{align}
\end{proposition}

\proof{} See Appendix.

\section{Deterministic IPDS Framework.}\label{sec:det_framework}

We develop the IPDS framework based on the idea of the {\em smoothed duality gap reduction}. 

\begin{assumption}\label{assum:bounded_Lambda}
The set $\Lambda$ is bounded. 
\end{assumption}

To see the implication of this assumption, for any $(x,\lambda)\in\calX\times\Lambda$, we may bound  

\begin{align}
\abst{\Delta(x,\lambda) - \Delta_\rho(x,\lambda)} &= \abst{{\sup}_{\lambda\in\Lambda}S(x,\lambda) - {\sup}_{\lambda\in\Lambda}S_\rho(x,\lambda)}\label{eq:bound_abs_Delta}\\
& \le {\sup}_{\lambda\in\Lambda}\abst{S(x,\lambda) -S_\rho(x,\lambda)}\nn = \rho \;{\sup}_{\lambda\in\Lambda}\abs{\omega(\lambda)}< +\infty,\nn
\end{align}
where the last inequality  follows from Assumption~\ref{assum:bounded_Lambda}, the closedness of $\Lambda$ and the continuity of $\omega$ on $\Lambda$. For later use, define 

\begin{equation}
\label{eq:Bomega}
B_{\omega,\Lambda}\defeq {\sup}_{\lambda\in\Lambda}\abs{\omega(\lambda)}. 
\end{equation} 

Since $\omega$ is also 1-s.c.\ on $\Lambda$, we conclude that $B_{\omega,\Lambda}<+\infty$ if and only if $\Lambda$ is bounded. 

\begin{remark}\label{rmk:assump}
We provide a few remarks about Assumption~\ref{assum:bounded_Lambda} (which is equivalent to $B_{\omega,\Lambda}<+\infty$). First, 
it helps to connect the smoothed duality gap $\Delta_\rho$ to the duality gap $\Delta$. Indeed, in our analysis, we  first analyze the convergence rate of the smoothed duality gap, and then show that the same rate holds for the (original) duality gap if $B_{\omega,\Lambda}<+\infty$. Note that we {\em do not} need this assumption to derive any convergence results regarding the smoothed duality gap. Finally, we note that Assumption~\ref{assum:bounded_Lambda} is required by many other algorithms for solving SPPs, e.g., Mirror-Prox~\cite{Nemi_05}, HPE-type~\cite{Kolo_17} and PDHG~\cite{Zhao_19}, although for different reasons. Indeed, it is typical in works where the duality gap is used as the convergence criterion, and is not specific to our work. 
\end{remark}

\subsection{Framework Description.}\label{sec:description}

\setlength{\textfloatsep}{1em}
\begin{algorithm}[t] 
\caption{Deterministic IPDS 
framework}
\begin{algorithmic}
\State {\bf Input}: Initial smoothing parameter $\rho_0>0$, nonnegative error sequences $\left\{ \eta_{k}\right\} _{k\in\bbZ_+}$ and $\left\{ \gamma_{k}\right\} _{k\in\bbZ_+}$, interpolation sequence $\left\{\tau_k\right\}_{k\in\bbZ_+}\subseteq (0,1)$ and  deterministic first-order algorithms $\rvN_1$ and $\rvN_2$.
\State {\bf Initialize}: $x^{0}\in \calX$,  $\lambda^{0}\in\Lambda$  and $k = 0$.
\State {\bf Repeat} (until some convergence criterion is met)
\begin{enumerate}
\item Use $\rvN_1$ to find $\tlambda_{\rho_k,\eta_{k}}(x^{k})\in\Lambda$ such that \label{step:approx_lambda_1}
\begin{equation}
\hpsi_{\rho_{k}}^{\rmP}(x^{k})-\hatS^\rmD_{\rho_{k}}(x^{k},\tlambda_{\rho_k,\eta_{k}}(x^{k}))\leq\eta_{k}.\label{eq:approx_lambda_1}
\end{equation}
\item Set $\hat{\lambda}^{k}=\tau_{k}\lambda^{k}+(1-\tau_{k})\tlambda_{\rho_k,\eta_{k}}(x^{k})$. \label{step:interp_dual1}
\item Use $\rvN_2$ to find $\tilx_{\gamma_{k}}(\hat{\lambda}^{k})\in \calX$ such that \label{step:approx_x}

\begin{equation}
\hatS^\rmP(\tilx_{\gamma_{k}}(\hat{\lambda}^{k}),\hat{\lambda}^{k})-\hpsi^\rmD(\hat{\lambda}^{k})\leq\gamma_{k}.\label{eq:approx_x}
\end{equation}
\item Set $x^{k+1}=\tau_{k}x^{k}+(1-\tau_{k})\tilx_{\gamma_{k}}(\hat{\lambda}^{k})$.\label{step:interp_primal}
\item Set $\rho_{k+1}=\tau_{k}\rho_{k}$. \label{step:update_mu}
\item Use $\rvN_1$ to find $\tlambda_{\rho_{k+1},\eta_{k}}(x^{k+1})\in\Lambda$ such that \label{step:approx_lambda_2}
\begin{equation}
\hpsi_{\rho_{k+1}}^{\rmP}(x^{k+1})-\hatS^\rmD_{\rho_{k+1}}(x^{k+1},\tlambda_{\rho_{k+1},\eta_{k}}(x^{k+1}))\leq\eta_{k}. \label{eq:approx_lambda_2}
\end{equation}
\item Set $\lambda^{k+1}=\tau_{k}\lambda^{k}+(1-\tau_{k})\tlambda_{\rho_{k+1},\eta_{k}}(x^{k+1})$.\label{step:interp_dual2}
\item Set $k = k + 1$.
\end{enumerate}
\State {\bf Output}: $(x^{\rm out},\lambda^{\rm out})\defeq(x^k,\lambda^k)$.
\end{algorithmic}
\label{alg:SMA}
\end{algorithm}
The framework is presented in Algorithm~\ref{alg:SMA}. We choose both $\rvN_1$ and $\rvN_2$ to be first-order methods. \revise{Let us describe the main ideas behind this framework}. 
From~\eqref{eq:bound_abs_Delta}, we observe that $\Delta(x,\lambda)\le \Delta_\rho(x,\lambda)+\rho B_{\omega,\Lambda}$. Therefore, if there exists a primal-dual pair $(x,\lambda)\in\calX\times\Lambda$ such that the smoothed duality gap $\Delta_\rho(x,\lambda)$ is small, then with a small smoothing parameter $\rho$, the duality gap $\Delta(x,\lambda)$  will also be small. This leads us to develop a framework that ``sufficiently'' reduces both the smoothed duality gap  and smoothing parameter in each iteration. Indeed, in step~\ref{step:update_mu} of Algorithm~\ref{alg:SMA}, the smoothing parameter $\rho_k$ is reduced by a factor of $\tau_k\in(0,1)$. The same factor is also used to do interpolation of the primal and dual iterates (cf.\ steps~\ref{step:interp_dual1},~\ref{step:interp_primal} and~\ref{step:interp_dual2}). In contrast to the smoothing frameworks for  bilinear SPPs (e.g.,~\cite{Nest_05,Nest_05b}), our framework does not require the sub-problems~\eqref{eq:hdualF}  and~\eqref{eq:hprimalF_sm} to be solved exactly. Instead, we only need  inexact solutions satisfying certain accuracy criteria (involving the parameters $\gamma_k$ and $\eta_k$; cf.~\eqref{eq:approx_lambda_1},~\eqref{eq:approx_x} and~\eqref{eq:approx_lambda_2}). In principle, such solutions can be computed via any first-order method. (For the implementation details, we refer to Section~\ref{sec:solve_subproblem_det}.)  The success of our framework hinges upon the proper choices of $\tau_k$, $\gamma_k$ and $\eta_k$, which ensure the reduction of the smoothed duality gap $\Delta_{\rho_{k}}(x^k,\lambda^k)$ in each iteration, and simultaneously decrease $\rho_k$ (cf.\ Section~\ref{sec:conv_analysis_det}). 




\subsection{Solving Sub-problems.}\label{sec:solve_subproblem_det}

In Algorithm~\ref{alg:SMA}, it is important to solve the sub-problems in steps~\ref{step:approx_lambda_1},~\ref{step:approx_x} and~\ref{step:approx_lambda_2} inexactly in an efficient manner. 
As mentioned in Section~\ref{sec:def_1}, these optimization problems have composite forms, hence it is natural for us to employ optimal first-order algorithms to solve them.
Examples of such algorithms include  the accelerated proximal gradient method ({\sf APG}) in~\cite[Equation~(4.9)]{Nest_13}, the primal dual gradient method ({\sf PDG})  in~\cite[Algorithm 2]{Lan_18} and a variant of the fast iterative shrinkage-thresholding algorithm  ({\sf V-FISTA}) in~\cite[Chapter 10]{Beck2017}. {\sf V-FISTA} and {\sf APG} are developed for optimization problems in finite-dimensional real Hilbert spaces. As mentioned in~\cite{Nest_13}, the analysis of {\sf APG} can be extended to finite-dimensional real normed spaces in a standard way as  in~\cite{Nest_05}. 
 In the sequel, for the purpose of theoretical analysis, we  use  {\sf APG} as the sub-problem solver in steps~\ref{step:approx_lambda_1},~\ref{step:approx_x} and~\ref{step:approx_lambda_2}.

Next, we briefly review the convergence rate of {\sf APG}. Consider 
the following composite optimization problem:

\begin{equation}
{\min}_{u\in\calU} \big\{\Psi(u)\defeq\phi_1(u) + \phi_2(u)\big\},\label{eq:det_model_problem}
\end{equation}
where $\calU$ is a convex and compact set in a finite-dimensional real normed space $\bbU$, 
 $\phi_1$ and $\phi_2$ are CCP functions, $\phi_1$ is $L'$-smooth on $\calU$, $\Psi$ is $\mu'$-s.c.\ on $\calU$ ($ \mu' >0$) and $\phi_2$ admits an easily computable BPP on $\calU$ with DGF $\pi$ (cf.\ Section~\ref{sec:def_1}). Denote the unique solution of \eqref{eq:det_model_problem} by $u^*\in\calU$ and define $\kappa'\defeq L'/\mu'$. For any starting point $u^0\in\calU^o\defeq \calU\cap\inter\dom\pi$,
from~\cite[Theorem~6]{Nest_13} (see also~\cite[Section~5.1]{Nest_13}), we have 

\begin{align}
\label{eq:bound_APG}
\Psi(u^N) - \Psi(u^*) &\le (L'/4)(1+1/\sqrt{2\kappa'})^{-2(N-1)} D_{\calU,\pi}(u^0)\quad\forall\,N\in\bbN,
\end{align} where $u^N$ denotes the $N$-th iterate and $D_{\calU,\pi}(u^0)$ denotes the Bregman radius of $\calU$ measured at $u^0$. Precisely, it is defined as 

\begin{equation}
D_{\calU,\pi}(u^0)\defeq {\max}_{u\in \calU}D_\pi(u,u^0)<+\infty, \label{eq:diam_U}
\end{equation} 
where 

\begin{equation}
D_\pi(u,u^0)\defeq\pi(u)-\pi(u^0)-\iprod{\nabla \pi(u^0)}{u-u^0} \label{eq:Bregman_divergence}
\end{equation}
denotes the Bregman divergence between $u$ and $u^0$. (Note that the finiteness of $D_{\calU,\pi}(u^0)$ follows from the compactness of $\calU$ and the continuity of $\pi$ on $\calU$.) 
In other words, to find an $\varepsilon$-inexact solution of~\eqref{eq:det_model_problem}, the number of iterations of {\sf APG} (or equivalently, the number of first-order oracle calls made by {\sf APG}) does not exceed

{\small\begin{align}
\left\lceil\sqrt{\frac{\kappa'}{2}}\log\left(\frac{L'D_{\calU,\pi}(u^0)}{4\varepsilon}\right)\right\rceil + 1 = O\left(\sqrt{\kappa'}\log\left(\frac{L'}{\varepsilon}\right)\right). \label{eq:comp_APG}
\end{align}}
Therefore, based on~\eqref{eq:comp_APG}, 
to find an $\eta$-inexact solution of~\eqref{eq:hprimalF_sm} (cf.~\eqref{eq:inexact_dual_sln}), the number of dual oracle calls (cf.\ Section~\ref{sec:oracle_model}) made by $\rvN_1$ does not exceed 
{\small\begin{equation}
C_{\sf N_1} \defeq n\left\{\left\lceil\sqrt{\frac{L_{\lambda\lambda}}{2\rho}}\log\left(\frac{L_{\lambda\lambda}D_{\Lambda,\omega}(\lambda^0)}{4\eta}\right)\right\rceil + 1\right\} = O\left(n\sqrt{\frac{L_{\lambda\lambda}}{\rho}}\log\left(\frac{L_{\lambda\lambda}}{\eta}\right)\right), \label{eq:comp_N1}
\end{equation}}
where $\lambda^0\in\Lambda^o$ denotes the starting point of $\rvN_1$ and $D_{\Lambda,\omega}(\lambda^0)\defeq {\max}_{\lambda \in \Lambda}D_\omega(\lambda,\lambda^0)$. 
In addition, to find a $\gamma$-inexact solution of~\eqref{eq:hdualF} (cf.~\eqref{eq:inexact_primal_sln}), the number of primal oracle calls made by $\rvN_2$ does not exceed

{\small\begin{equation}
\begin{split}
C_{\sf N_2} &\defeq (n+1) \left\{\left\lceil\sqrt{\frac{L+L_{xx}}{2\mu}}\log\left(\frac{(L+L_{xx})D_{\calX,\baromega}(x^0)}{4\gamma}\right)\right\rceil + 1\right\}\label{eq:comp_N2}
\\& = O\left(n\sqrt{\kappa_\calX}\log\big(({L+L_{xx}})/{\gamma}\big)\right), 
\end{split}
\end{equation}}
where $x^0\in\calX^o$ denotes the starting point of $\rvN_2$ and $D_{\calX,\baromega}(x^0)\defeq {\max}_{x \in \calX}D_\omega(x,x^0)$. 
Note that by Assumption~\ref{assum:bounded_Lambda}, 
$D_{\Lambda,\omega}(\lambda^0)$ and $D_{\calX,\baromega}(x^0)$ are both finite, for any $\lambda^0\in\Lambda^o$ and $x^0\in\calX^o$. 
Therefore, we do {\em not} need to know the optimal solution or the optimal objective value of the  problem in~\eqref{eq:hprimalF_sm} (resp.~\eqref{eq:hdualF}) in order to find $\tlambda_{\rho,\eta}(x)$ (resp.~$\tilx_{\gamma}(\lambda)$). Instead, we can simply run {\sf APG} 
for a pre-determined number of iterations according to the iteration upper bounds in~\eqref{eq:comp_N1} and~\eqref{eq:comp_N2}.  

Finally, note that in Algorithm~\ref{alg:SMA}, if we fix $\barx\in\calX^o$ and $\barlambda\in\Lambda^o$ and use them as the starting points to solve~\eqref{eq:approx_x},~\eqref{eq:approx_lambda_1} and~\eqref{eq:approx_lambda_2} in iteration $k$, then both $D_{\calX,\baromega}(\barx)$ and $R_{\Lambda,\omega}(\barlambda)$ are finite constants independent of $k$. When the diameter of $\mathcal U$ can be estimated easily (as in Application I in Section~\ref{sec:intro} or the examples in~\cite[Section 4]{Nest_05}), it may be preferable in the practical sense to use the iterates from the current iteration as the starting point for solving the subproblems in the next iteration.     


\subsubsection{An adaptive stopping criterion} \label{sec:adaptive_crit}

Note that the aforementioned strategy to find $\tlambda_{\rho,\eta}(x)$ and~$\tilx_{\gamma}(\lambda)$, i.e., by running  {\sf APG} for a pre-determined number of iterations, has two potential drawbacks. First, in some cases, it is hard to find upper bounds for the Bregman radii $D_{\Lambda,\omega}(\lambda^0)$ and $D_{\calX,\baromega}(x^0)$. Second, even if we can find these upper bounds, for certain cases, the strategy can be overly conservative in practice, as the number of iterations given in~\eqref{eq:comp_N1} and~\eqref{eq:comp_N2} are obtained for the worst cases.    
Therefore, in the following, we provide an adaptive stopping criterion for $\rvN_1$ and $\rvN_2$ that possibly serves as a remedy for the drawbacks described above.

Note that for almost all proximal-gradient-type methods (including {\sf APG}, {\sf PDG} and {\sf V-FISTA}) applied to~\eqref{eq:det_model_problem}, each iteration $k$ involves solving the following sub-problem: 

\begin{equation}
\label{eq:mirror_descent_step}
  u^{k+1}={\argmin}_{u \in \mathcal U}\; \{\phi_2(u)+  \lranglet{\nabla\phi_1(\tilu^{k})}{u} + L'D_\pi(u,\tilu^k)\} \in\calU^o,
\end{equation} 
where  $\tilu^k$ may or may not be the same as $u^k$, depending on whether acceleration is involved. 
\color{black}
Let us compute an additional point $\bar u^{k+1}$ as follows:

\begin{equation}
 \label{eq:ghostpoint}
 \bar u^{k+1}={\argmin}_{u \in \mathcal U}\; \{\phi_2(u)+  \lranglet{\nabla\phi_1(u^{k+1})}{u} + L'D_\pi(u,u^{k+1})\} \in\calU^o.
\end{equation}
(Note that we assume both $u^{k+1}$ and $\bar u^{k+1}$ can be easily computed; cf.~Section~\ref{sec:def_1}.)  Define

\begin{equation}
G^{k+1}\defeq\nabla \phi_1(\bar u^{k+1}) - \nabla \phi_1(u^{k+1}) - L' \big(\nabla \pi(\bar u^{k+1})-\nabla \pi(u^{k+1})\big). \label{eq:def_G_k}
\end{equation} 
The first-order optimality condition for \eqref{eq:ghostpoint} reads 
 \begin{equation}
 \label{eq:opt_cond}
0\in \partial (\phi_2+  \iota_{\mathcal U})(\bar u^{k+1})+\nabla \phi_1( u^{k+1}) + L' \big(\nabla \pi(\bar u^{k+1})-\nabla \pi(u^{k+1})\big),  
 \end{equation}
 where $\iota_{\mathcal U}$ denotes the indicator function of $\calU$, i.e., $\iota_{\mathcal U}(u)=0$ if $u\in\calU$ and $\iota_{\mathcal U}(u)=+\infty$ otherwise. Combining~\eqref{eq:def_G_k} and ~\eqref{eq:opt_cond}, we see that $G^{k+1}\in \partial (\Psi+\iota_{\mathcal U})(\bar u^{k+1})$. On the other hand, as $\Psi$ is $\mu'$-s.c.\ on $\calU$, we have
\begin{equation}
\Psi(u^*)-\Psi(\bar u^{k+1})\geq \iprod{G^{k+1}}{u^*-\bar  u^{k+1}} + ({\mu'}/{2})\|u^*-\bar  u^{k+1}\|^2\geq -\|G^{k+1}\|_*^2/({2\mu'}),
\end{equation}
 where the last step follows from Young's inequality, i.e., $\abs{\iprod{a}{b}}\leq (s/2)\|a\|_*^2 + \|b\|^2/(2s) $, for any $s>0$, $a\in\bbU^*$ and $b\in\bbU$. Therefore, we obtain 
 \begin{equation}
 \label{eq:stopping_criteria}
 \Psi(\bar u^{k+1})-\Psi(u^*) \leq \|G^{k+1}\|_*^2/({2\mu'}). 
 \end{equation}
 From~\eqref{eq:stopping_criteria}, we see that to find an $\varepsilon$-inexact solution of~\eqref{eq:det_model_problem} using any proximal-gradient-type method, it suffices to stop at iteration $k$ where $\normt{G^k}_*\le \sqrt{2\mu'\varepsilon}$.

Next, let us suppose that the DGF $\pi$ has $L_\pi$-Lipschitz gradient on $\calU$, i.e.,
\begin{equation}
\|\nabla\pi(u_1)-\nabla\pi(u_2)\|_* \leq L_\pi \|u_1 - u_2\|, \quad\forall\,u_1,u_2\in\calU.\label{eq:smooth_pi}
\end{equation}
(Note that this happens if $\calU\subseteq \inter \dom \pi$ or if the normed space $\bbU$ is Hilbertian with inner product $\lranglet{\cdot}{\cdot}$ and its induced norm $\normt{\cdot}$, 
and $\pi(\cdot)=(1/2)\normt{\cdot}^2$.)  Based on this assumption, we show that 
 if an optimal first-order algorithm is used to solve Problem~\eqref{eq:det_model_problem}, then the condition $\normt{G^k}_*\le \sqrt{2\mu'\varepsilon}$ will be satisfied after $O(\sqrt{\kappa'}\log({1}/{\varepsilon}))$ iterations. 
 It follows from~\eqref{eq:ghostpoint} and \cite[Property 1]{Tseng_08} that 
 \begin{equation} 
 \label{eq:temp1}
\begin{split}
&\phi_2(u^{k+1}) +  \lranglet{\nabla\phi_1(u^{k+1})}{u^{k+1}}\\  
&\geq \phi_2(\baru^{k+1}) +  \lranglet{\nabla\phi_1(u^{k+1})}{\bar u^{k+1}} + L'D_\pi(\bar u^{k+1},u^{k+1}) + L'D_\pi(u^{k+1},\bar u^{k+1}). 
\end{split} 
\end{equation}
Furthermore, since $\phi_1$ is $L'$-smooth and $\pi$ is 1-s.c.\ on $\mathcal U$, we have 
\begin{equation}
\label{eq:temp2}
\phi_1(\bar u^{k+1})\leq \phi_1(u^{k+1}) + \lranglet{\nabla\phi_1(u^{k+1})}{\bar u^{k+1}-u^{k+1}} + (L'/2) \|u^{k+1}-\bar u^{k+1}\|^2,
\end{equation}
\begin{equation}
\label{eq:temp3}
\min\{D_\pi(\bar u^{k+1},u^{k+1}), D_\pi(u^{k+1},\bar u^{k+1})\} \geq (1/2) \|u^{k+1}-\bar u^{k+1}\|^2. 
\end{equation}
From \eqref{eq:temp1}, \eqref{eq:temp2} and \eqref{eq:temp3}, we obtain 
\begin{equation*}
\Psi(u^{k+1})-\Psi(u^*)\geq \Psi(\bar u^{k+1})-\Psi(u^*) + ({L'}/{2})\|u^{k+1}-\bar u^{k+1}\|^2.
\end{equation*}
Hence, we have 
\begin{equation}
\label{eq:normbounded}
\|u^{k+1}-\bar u^{k+1}\|^2 \leq ({2}/{L'})(\Psi(u^{k+1})-\Psi(u^*)).
\end{equation}
Furthermore, from~\eqref{eq:def_G_k} and~\eqref{eq:smooth_pi}, we have $ \|G^{k+1}\|_* \leq L'(1+L_\pi  )\|u^{k+1}-\bar u^{k+1}\|$. Together with~\eqref{eq:normbounded}, this implies  
\begin{equation}
\label{eq:Gkbounded}
\|G^{k+1}\|^2_* \leq 2L'(1+L_\pi)^2(\Psi(u^{k+1})-\Psi(u^*)).
\end{equation}
From~\eqref{eq:comp_APG}, we see that if an optimal first-order method is used to solve Problem~\eqref{eq:det_model_problem}, then \revise{$O(\sqrt{\kappa'}\log({L'}/{\varepsilon'}))$} iterations are needed to guarantee  $\Psi(u^{k})-\Psi(u^*) \leq \varepsilon'$, where
\revise{\mbox{$
\varepsilon'=\frac{2\mu' \varepsilon}{L' (1+L_\pi)^2}. 
$}}
 \revise{ Therefore, \eqref{eq:Gkbounded} implies that the inequality $\normt{G^k}_*\le \sqrt{2\mu'\varepsilon}$ will be satisfied after $O\big(\sqrt{\kappa'}\log(\frac{\kappa'(1+L_\pi)^2 L'}{2\varepsilon})\big) $ iterations, which has the same order of dependence on $\varepsilon$ as the complexity  $O\big(\sqrt{\kappa'}\log(\frac{L'}{\varepsilon})\big)$ obtained by the strategy of running  {\sf APG} a pre-determined number of iterations (in fact, $O\big(\sqrt{\kappa'}\log(\frac{\kappa'(1+L_\pi)^2 L'}{2\varepsilon})\big) \approx O\big(\sqrt{\kappa'}\log(\frac{L'}{\varepsilon})\big)$). In other words, if we use the adaptive stopping criterion, we will obtain the oracle complexities that have the same order of  $\varepsilon$ with the oracle complexities established in Section \ref{sec:oracle_comp_det}.  }
  
Before concluding this section, we remark that the $L_\pi$-smoothness condition of the DGF $\pi$ on $\calU$ (cf.~\eqref{eq:smooth_pi}) is important in the above derivation, and extension to DGFs without such a smoothness property is left to future work. 
  
\color{black}

\subsection{Convergence Analysis.}\label{sec:conv_analysis_det}

As the first step, we prove that 
in each iteration $k$, if the smoothing parameter $\rho_k$ is chosen to be sufficiently large, then  the smoothed duality gap is reduced. 

\begin{lemma}\label{lem:SmDuGap}
In Algorithm~\ref{alg:SMA}, for any $k\in\bbZ_+$, if ${\rho_{k+1}}\geq {4(1-\tau_{k})^{2}}L_\rmD$, %
then 
\begin{equation}
\Delta_{\rho_{k+1}}(x^{k+1},\lambda^{k+1})\leq\tau_{k}\Delta_{\rho_{k}}(x^k,\lambda^k)+2\gamma_{k}+2\eta_{k}. \label{eq:recur_SmDuGap}
\end{equation} 
\end{lemma}
\proof{} See Appendix. 

 In Lemma~\ref{lem:SmDuGap}, we notice that if $\Delta_{\rho_{k}}(x^k,\lambda^k)> 2(\gamma_k+\eta_k)/(1-\tau_k)$, then the smoothed duality gap will be reduced, i.e., $\Delta_{\rho_{k+1}}(x^{k+1},\lambda^{k+1})< \Delta_{\rho_{k}}(x^k,\lambda^k)$. Indeed, from our choices of $\tau_k$, $\gamma_k$ and $\eta_k$ in Theorem~\ref{thm:main} (see below), the reduction holds as long as $\Delta_{\rho_{k}}(x^k,\lambda^k)> \varepsilon/2$. This corroborates our description in Section~\ref{sec:description}. 

Before proving our main convergence results,  let us state a result about linear recursion, whose proof simply follows from induction.  
\begin{lemma}\label{lem:lin_recursion}
Let $\{\alpha_k\}_{k\in\bbZ_+}$, $\{\beta_k\}_{k\in\bbZ_+}$ and $\{a_k\}_{k\in\bbZ_+}$ be real sequences.  
If for all $k\in\bbZ_+$, $\alpha_k\ge 0$,  
\begin{equation}
a_{k+1}\le \alpha_k a_{k}+\beta_k,  
\end{equation} 
then for all $K\in \bbN$, 
\begin{equation}
a_K\le \left(\textstyle{\prod}_{k=0}^{K-1}\,\alpha_k\right)a_0 + \textstyle{\sum}_{k=1}^{K}\left(\textstyle{\prod}_{j=k}^{K-1}\,\alpha_{j}\right)\beta_{k-1}, 
\end{equation}
where we define the empty product $\prod_{j=K}^{K-1}\alpha_j\defeq1$. 
\end{lemma}

Based on Lemmas~\ref{lem:SmDuGap} and~\ref{lem:lin_recursion}, our main results follow immediately. 

\begin{theorem}\label{thm:main} 
In Algorithm~\ref{alg:SMA}, if we choose $\rho_{0}={8L_\rmD}$ and  for any $k\in\bbZ_+$,
\begin{equation}
\tau_{k}=\frac{k+1}{k+3},\quad \gamma_{k}=\frac{\varepsilon}{4(k+3)} \quad\mbox{and}\quad \eta_{k}=\frac{\varepsilon}{4(k+3)}, \label{eq:choose_param_SMA}
\end{equation}
then for any starting point  $(x^0,\lambda^0)\in\calX\times \Lambda$ and $K\in\bbN$, 
\begin{equation}
\Delta_{\rho_{K}}(x^K,\lambda^K)\leq B'_\Delta(K,\varepsilon)\defeq \frac{2\Delta_{\rho_0}(x^0,\lambda^0)}{(K+1)(K+2)}+\frac{\varepsilon}{2}. 
\label{eq:Delta_k_sm_conv_rate}
\end{equation}
Furthermore, if Assumption~\ref{assum:bounded_Lambda} holds, then 
\begin{equation}
\label{eq:Delta_k_conv_rate}
\Delta(x^K,\lambda^K)\leq B_\Delta(K,\varepsilon)\defeq\frac{32L_\rmD B_{\omega,\Lambda}+2\Delta(x^0,\lambda^0)}{(K+1)(K+2)}+\frac{\varepsilon}{2}. 
\end{equation}
\end{theorem}

\proof{}
Based on the choice of $\rho_0$ and $\{\tau_k\}_{k\in\bbZ_+}$, for any $K\in\bbN$, we have
\begin{equation}
{\prod}_{k=0}^{K-1}\tau_{k} = \frac{2}{(K+1)(K+2)}\quad\Longrightarrow\quad \rho_{K}=\rho_0{\prod}_{k=0}^{K-1}\tau_{k} = \dfrac{16L_\rmD}{(K+1)(K+2)}. \label{eq:expression_rhoK} 
\end{equation}
Therefore, we can easily verify that the condition ${\rho_{K}}\geq {4(1-\tau_{K-1})^{2}}L_\rmD$ in Lemma~\ref{lem:SmDuGap} is satisfied. Consequently, by the recursion in~\eqref{eq:recur_SmDuGap} and Lemma~\ref{lem:lin_recursion}, we have
\begin{align*}
\Delta_{\rho_{K}}(x^K,\lambda^K)\leq & \Delta_{\rho_0}(x^0,\lambda^0){\prod}_{k=0}^{K-1}\tau_{k}+{\sum}_{k=1}^{K}2(\gamma_{k-1}+\eta_{k-1}){\prod}_{j=k}^{K-1}\tau_{j}\nt\label{eq:mainProof}\\
= & \frac{2\Delta_{\rho_0}(x^0,\lambda^0)}{(K+1)(K+2)}+{\sum}_{k=1}^{K}\;\frac{\varepsilon}{k+2}\cdot\frac{(k+1)(k+2)}{(K+1)(K+2)}\\
=& \frac{2\Delta_{\rho_0}(x^0,\lambda^0)}{(K+1)(K+2)}+\frac{\varepsilon}{2}\frac{K(K+3)}{(K+1)(K+2)}. 
\end{align*}
We then obtain~\eqref{eq:Delta_k_sm_conv_rate} by noticing that $K(K+3)\le (K+1)(K+2)$. 
Based on~\eqref{eq:Delta_k_sm_conv_rate}, to obtain~\eqref{eq:Delta_k_conv_rate}, we simply use \eqref{eq:bound_abs_Delta} and~\eqref{eq:expression_rhoK}.  
~\hfill\qed\endproof

\begin{remark}
From~\eqref{eq:Delta_k_conv_rate}, since $L_\rmD$ only depends on $L_{\lambda x}$ and $L_{\lambda\lambda}$ (cf.~\eqref{eq:L_D}), we note that the convergence of the duality gap in Algorithm~\ref{alg:SMA} only requires the Lipschitz continuity of $\nabla_\lambda \Phi(\cdot,\lambda)$ and $\nabla_\lambda \Phi(x,\cdot)$ (cf.~\eqref{eq:Phi_sm(c)} and~\eqref{eq:Phi_sm(d)}), but not the Lipschitz continuity of $\nabla f$, $\nabla_x \Phi(\cdot,\lambda)$ and $\nabla_x \Phi(x,\cdot)$ (cf.~\eqref{eq:sm_sc_f}, \eqref{eq:Phi_sm(a)} and~\eqref{eq:Phi_sm(b)}).  However, the latter smoothness conditions are needed in order to use an optimal first-order algorithm, as introduced in Section~\ref{sec:solve_subproblem_det}, to solve the sub-problems in~\eqref{eq:approx_x}. By doing so, in Algorithm~\ref{alg:SMA}, we can achieve the overall primal oracle complexity $\tilO(1/\sqrt{\varepsilon})$. For details, we refer readers to Section~\ref{sec:oracle_comp_det}. 
\end{remark}

Note that Theorem~\ref{thm:main} indicates that in Algorithm~\ref{alg:SMA}, to achieve an $\varepsilon$-duality gap, the number of iterations we need is 
{\small\begin{align}
K_{\sf det} \defeq \left\lceil\frac{2\sqrt{16L_\rmD B_{\omega,\Lambda}+\Delta(x^0,\lambda^0)}}{\sqrt{\varepsilon}}\right\rceil +1 = O\left(\sqrt{\frac{L_\rmD}{\varepsilon}}\right). \label{eq:number_iter_det}
\end{align}}
By the definition of $L_\rmD$ in~\eqref{eq:L_D}, we have $K_{\sf det}=O(\sqrt{L_{\lambda\lambda}/{\varepsilon}}+{L_{\lambda x}}/{\sqrt{\mu\varepsilon}})$. 

\subsection{Oracle Complexity.}\label{sec:oracle_comp_det}
Based on the results in Sections~\ref{sec:solve_subproblem_det} and~\ref{sec:conv_analysis_det} (specifically,~\eqref{eq:comp_N1},~\eqref{eq:comp_N2} and~\eqref{eq:number_iter_det}), we may analyze the primal and dual oracle complexities needed in Algorithm~\ref{alg:SMA} to achieve an $\varepsilon$-duality gap (i.e., $\Delta(x^{\rm out},\lambda^{\rm out})\le \varepsilon$). 

\begin{theorem}\label{thm:det_comp}
Let Assumption~\ref{assum:bounded_Lambda} hold.  In Algorithm~\ref{alg:SMA}, for any starting point $(x^0,\lambda^0)\in\calX\times\Lambda$, let $C_{\sf det}^\rmP$ and $C_{\sf det}^\rmD$ denote the primal and dual oracle complexities (cf.~Section~\ref{sec:oracle_model})  to achieve an $\varepsilon$-duality gap, respectively. Then, we have 
\begin{align}
C_{\sf det}^\rmP & = O\left(n\sqrt{{\kappa_\calX L_\rmD}/{{\varepsilon}}}\log\big({{(L+L_{xx})L_\rmD}}/{\varepsilon}\big)\right),\\
C_{\sf det}^\rmD & = O\left(n\big({\sqrt{L_{\lambda \lambda}L_\rmD}}/{\varepsilon}\big)\log\left({L_{\lambda\lambda}L_{\rmD}}/{\varepsilon}\right)\right).
\end{align}
\end{theorem}
\proof{} See Appendix. 

\begin{remark}\label{rmk:special_affine_case}
If $\Phi(x,\cdot)$ is an affine function for any $x\in\calX$, Problem~\eqref{eq:hprimalF_sm} has the same form as the BPP in~\eqref{eq:BPP}. By our assumption that $h$ has a tractable BPP on $\Lambda$, the problem in~\eqref{eq:hprimalF_sm} therefore admits an easily computable solution (and can be solved exactly). 
In this case, the number of oracle calls made by $\rvN_1$ in steps~\ref{step:approx_lambda_1} and~\ref{step:approx_lambda_2} of Algorithm~\ref{alg:SMA} clearly has order $O(1)$ with respect to $\varepsilon$. As a result, in terms of dependence on $\varepsilon$, the dual oracle complexity $C_{\sf det}^\rmD$ has the same order as $K_{\sf det}$ in~\eqref{eq:number_iter_det}, which is $O(1/\sqrt{\varepsilon})$.  
The same comments also apply to the dual oracle complexity  of Algorithm~\ref{alg:SRSA} in Section~\ref{sec:rand_framework} (see below). 

\end{remark}

\section{Randomized IPDS Framework.}\label{sec:rand_framework}

\begin{algorithm}[t]
\caption{Randomized primal-dual smoothed gap reduction framework}
\begin{algorithmic}
\State {\bf Input}: Initial smoothing parameter $\rho_0>0$, 
nonnegative error sequences $\left\{ \eta_{k}\right\} _{k\in\bbZ_+}$ and $\left\{ \gamma_{k}\right\} _{k\in\bbZ_+}$, interpolation sequence $\left\{\tau_k\right\}_{k\in\bbZ_+}\subseteq (0,1)$ and  
randomized first-order algorithms $\rvM_{1}$ and $\rvM_{2}$.
\State {\bf Initialize}: $x^{0}\in \calX$,  $\lambda^{0}\in\Lambda$  and $k = 0$.
\State {\bf Repeat} (until some convergence criterion is met)
\begin{enumerate}
\item Use $\rvM_{1}$ to find $\tlambda_{\rho_k,\eta_{k}}(x^{k})\in\Lambda$ \label{step:approx_lambda_1_stoc}
such that 
\begin{equation}\label{eq:approx_lambda_1_stoc}
\bbE\big[\psi_{\rho_{k}}^{\rmP}(x^{k})-S_{\rho_{k}}(x^{k},\tlambda_{\rho_k,\eta_{k}}(x^{k}))\,\big\vert\,\calF_{k,0}\big]\leq\eta_{k} \quad \mbox{a.s.}
\end{equation}
\item Set $\hat{\lambda}^{k}=\tau_{k}\lambda^{k}+(1-\tau_{k})\tlambda_{\rho_k,\eta_{k}}(x^{k})$.
\item Use $\rvM_{2}$ to find $\tilx_{\gamma_{k}}(\hat{\lambda}^{k})\in \calX$\label{step:approx_x_stoc}
such that 
\begin{equation}\label{eq:approx_x_stoc}
\bbE\big[ S(\tilx_{\gamma_{k}}(\hat{\lambda}^{k}),\hat{\lambda}^{k})-\psi^\rmD(\hat{\lambda}^{k})\,\big\vert\,\calF_{k,1}\big]\leq\gamma_{k} \quad \mbox{a.s.} 
\end{equation}
\item Set $x^{k+1}=\tau_{k}x^{k}+(1-\tau_{k})\tilx_{\gamma_{k}}(\hat{\lambda}^{k})$.
\item Set $\rho_{k+1}=\tau_{k}\rho_{k}$.
\item Use $\rvM_{1}$ to find $\tlambda_{\rho_{k+1},\eta_{k}}(x^{k+1})\in\Lambda$\label{step:approx_lambda_2_stoc}
such that 
\begin{equation}\label{eq:approx_lambda_2_stoc}
\bbE\big[\psi_{\rho_{k+1}}^{\rmP}(x^{k+1})-S_{\rho_{k+1}}(x^{k+1},\tlambda_{\rho_{k+1},\eta_{k}}(x^{k+1}))\,\big\vert\,\calF_{k,2}\big]\leq\eta_{k} \quad \mbox{a.s.}
\end{equation}
\item Set $\lambda^{k+1}=\tau_{k}\lambda^{k}+(1-\tau_{k})\tlambda_{\rho_{k+1},\eta_{k}}(x^{k+1})$.
\item Set $k = k + 1$.
\end{enumerate}
\State {\bf Output}: $(x^k,\lambda^k)$.
\end{algorithmic}
\label{alg:SRSA}
\end{algorithm}

When  $\Phi(\cdot,\cdot)$ has a large number of components \revise{(that is, when $n$ in \eqref{eq:finite-sum} is large)}, 
 we propose to  find the inexact solutions in steps~\ref{step:approx_lambda_1},~\ref{step:approx_x} and~\ref{step:approx_lambda_2} of Algorithm~\ref{alg:SMA} using randomized first-order methods. 
This is because randomized first-order methods, especially those incorporating the variance-reduction techniques (e.g.,~\cite{Shalev_16,Lan_18}), enjoy superior oracle complexities compared to their deterministic counterparts, for solving finite-sum  convex composite problems. 
Based on this idea, we develop our randomized IPDS framework, which is shown in Algorithm~\ref{alg:SRSA}. 

Note that at each iteration $k$, in steps~\ref{step:approx_lambda_1_stoc},~\ref{step:approx_x_stoc} and~\ref{step:approx_lambda_2_stoc} of Algorithm~\ref{alg:SRSA}, the inexact solutions that we aim to find are functions of some {\em stochastic} iterates, i.e., $x^k$, $\hlambda^k$ and $x^{k+1}$. Therefore, to analyze such inexact solutions, we need to properly condition on the past information. To this end, let us denote the  probability space for all the stochastic processes in Algorithm~\ref{alg:SRSA} by $(\Omega,\calB,\Pr)$ (where $\calB$ denotes the Borel $\sigma$-field on $\Omega$) and define a filtration $\bigcup_{k\in\bbZ_+}\{\calF_{k,i}\}_{i=0}^2$, where $\calF_{0,0}\defeq\{\emptyset,\Omega\}$ 
and for any $k\in\bbZ_+$,  
 
\begin{align*}
\calF_{k,1} &\defeq \sigma\big\{\calF_{k,0}\cup\sigma\big\{\tlambda_{\rho_k,\eta_{k}}(x^{k})\big\}\big\}, \quad
\calF_{k,2} \defeq \sigma\big\{\calF_{k,1}\cup\sigma\big\{\tilx_{\gamma_{k}}(\hat{\lambda}^{k})\big\}\big\},\\
\calF_{k+1,0} &\defeq \sigma\big\{\calF_{k,2}\cup\sigma\big\{\tlambda_{\rho_{k+1},\eta_{k}}(x^{k+1})\big\}\big\}. 
\end{align*}
Here we overload the notation $\sigma\{\cdot\}$ to represent the $\sigma$-field generated by either a family of  (Borel-measurable) sets or a random variable. From this definition, we clearly have 
\begin{equation}
\calF_{k,0}\subseteq\calF_{k,1}\subseteq\calF_{k,2}\subseteq\calF_{k+1,0}, \quad \forall\,k\in\bbZ_+.\label{eq:nested}
\end{equation}  
For any random variable $\xi$ and $\sigma$-field $\calF$, let $\xi\in\calF$ indicate that $\xi$ is measurable w.r.t.\ $\calF$. 
Then, we have $x^0,\lambda^0\in\calF_{0,0}$ and for any $k\in\bbZ_+$, 
\begin{align}
 \tlambda_{\rho_k,\eta_{k}}(x^{k}), \hlambda^k\in\calF_{k,1}, \;\;  \tilx_{\gamma_{k}}(\hat{\lambda}^{k}),x^{k+1}\in\calF_{k,2},\;\;\tlambda_{\rho_{k+1},\eta_{k}}(x^{k+1}),\lambda^{k+1}\in\calF_{k+1,0}.
 \label{eq:measurable_RV}
\end{align}

\subsection{Solving Sub-problems.}\label{sec:solve_subproblem_stoc}

Similar to the deterministic case (cf.\ Section~\ref{sec:solve_subproblem_det}), we choose both $\rvM_1$ and $\rvM_2$ to be first-order randomized methods. Examples of such methods include the stochastic variance-reduced gradient method ({\sf SVRG})~\cite{LinZhang2014} and the randomized primal-dual gradient method ({\sf RPD}) in~\cite[Algorithm~3]{Lan_18}. 
For our purpose, we  choose both $\rvM_1$ and $\rvM_2$ to be {\sf RPD}. 
Consider the optimization problem~\eqref{eq:det_model_problem}, where  $\phi_1$ has a finite-sum structure, i.e.,
\begin{equation}
\phi_1 (u)= (1/m)\textstyle{\sum}_{i=1}^m \varphi_i(u)\label{eq:finite-sum-subproblem}
\end{equation}
and each $\varphi_i$ is convex  and $L'_i$-smooth on $\calU$ (so that $L'\le (1/m){\sum}_{i=1}^n L'_i$). From the convergence results in~\cite[Corollary~1]{Lan_18}, for any starting point $u^0\in\calU^o$, to have $\bbE[\Psi(\tilu^N)-\Psi(u^*)]\le \varepsilon$, where $\tilu^N$ denotes the $N$-th iterate of {\sf RPD}, it suffices to let 
\begin{align*}
N &= 2(m+\sqrt{8m\kappa'})\log\left(2(L'/\sqrt{\mu'}+\sqrt{\mu'})^2(m+\sqrt{8m\kappa'})D_\pi(u^*,u^0)/\varepsilon\right)\nt\label{eq:comp_stoc_solver}\\
&\le 2(m+\sqrt{8m\kappa'})\log\left(2(L'/\sqrt{\mu'}+\sqrt{\mu'})^2(m+\sqrt{8m\kappa'})D_{\calU,\pi}(u^0)/\varepsilon\right)\\
&= O\big((m+\sqrt{m\kappa'})\log(L'\kappa'(m+\sqrt{m\kappa'})D_{\calU,\pi}(u^0)/\varepsilon)\big), 
\end{align*} 
where 
$D_{\calU,\pi}(u^0)<+\infty$ is defined in~\eqref{eq:diam_U}. Thus, similar to the {\sf APG} algorithm, to find an $\varepsilon$-inexact solution of~\eqref{eq:det_model_problem} in expectation using {\sf RPD}, we do not need to know $u^*$ or $\Psi(u^*)$. Instead, we simply run {\sf RPD} for a pre-determined number of iterations according to~\eqref{eq:comp_stoc_solver}. 

Based on~\eqref{eq:comp_stoc_solver}, for {\em any} $x\in\calX$, if $C_{\rvM_1}$ denotes the number of dual oracle calls of $\rvM_1$ to find an $\eta$-inexact solution of~\eqref{eq:hprimalF_sm} in expectation, i.e., $\tlambda_{\rho,\eta}(x)$ such that $\bbE[\hpsi^\rmP_\rho(\lambda) - \hatS^\rmD_\rho(x,\tlambda_{\rho,\eta}(x))]\le \eta$, then 
\begin{equation}
C_{\rvM_1} = O\left(\left(n+\sqrt{{nL_{\lambda\lambda}/\rho}}\right)\log\left({L_{\lambda\lambda}\big(n+\sqrt{nL_{\lambda\lambda}/\rho}\big)}/({\rho\eta})\right)\right).\label{eq:comp_M1}
\end{equation}
Similarly, for {\em any} $\lambda\in\Lambda$, if $C_{\rvM_2}$ denotes the number of primal oracle calls to $\rvM_2$ to find a $\gamma$-inexact solution of~\eqref{eq:hdualF}, i.e., $\tilx_\gamma(\lambda)$ such that $\bbE[\hatS^\rmP(\tilx_\gamma(\lambda),\lambda) - \hpsi^\rmD(\lambda)]\le \gamma$,  then 
\begin{equation}
C_{\rvM_2} = O\big((n+\sqrt{n\kappa_\calX})\log\big({(L+L_{xx})(n+\sqrt{n\kappa_\calX})}/{(\mu\gamma)}\big)\big).\label{eq:comp_M2}
\end{equation}



\subsection{Convergence Analysis.}\label{sec:conv_analysis_stoc}

We analyze the convergence rate of Algorithm~\ref{alg:SRSA} in expectation. For convergence results w.h.p., we refer readers to Section~\ref{sec:conv_whp}.  

\begin{lemma}\label{lem:SmDuGap_stoc}
In Algorithm~\ref{alg:SRSA}, for any $k\in\bbZ_+$, if ${\rho_{k+1}}\geq {4(1-\tau_{k})^{2}}L_\rmD$,
then 
\begin{equation}\label{eq:SmDuGap_stoc}
\bbE[\Delta_{\rho_{k+1}}(x^{k+1},\lambda^{k+1})\,\vert\,\calF_{k,0}]\leq\tau_{k}\Delta_{\rho_{k}}(x^k,\lambda^k)+2\gamma_{k}+2\eta_{k} \quad \mbox{a.s.}
\end{equation}
\end{lemma}

\proof{} See Appendix. 

Based on Lemma~\ref{lem:SmDuGap_stoc}, we can derive the convergence rate of Algorithm~\ref{alg:SRSA} in expectation. The proof directly follows that of Theorem~\ref{thm:main} and the tower property of conditional expectation, hence it is omitted. 

\begin{theorem} \label{thm:main_stoc}
In Algorithm~\ref{alg:SRSA}, if we choose the input parameters $\rho_{0}$, $\{\tau_k\}_{k\in\bbZ_+}$, $\{\gamma_k\}_{k\in\bbZ_+}$ and $\{\eta_k\}_{k\in\bbZ_+}$ in the same way as in Theorem~\ref{thm:main},  then for any starting point  $(x^0,\lambda^0)\in\calX\times \Lambda$ and $K\in\bbN$, $\bbE[\Delta_{\rho_{K}}(x^K,\lambda^K)]\leq B'_\Delta(K,\varepsilon)$ $($defined in~\eqref{eq:Delta_k_sm_conv_rate}$)$. 
Moreover, if Assumption~\ref{assum:bounded_Lambda} holds, then  $\bbE[\Delta(x^K,\lambda^K)]\leq B_\Delta(K,\varepsilon)$ $($defined in~\eqref{eq:Delta_k_conv_rate}$)$. 
\end{theorem}

Denote $K_{\sf stoc}$ as the number of iterations needed to achieve an $\varepsilon$-expected duality gap in Algorithm~\ref{alg:SRSA}. Based on Theorem~\ref{thm:main_stoc},  we have that  $K_{\sf stoc}=K_{\sf det}=O(\sqrt{L_\rmD/\varepsilon})$.

\subsection{Oracle Complexity.}\label{sec:oracle_comp_stoc}

We analyze the primal and dual oracle complexities of Algorithm~\ref{alg:SRSA} to achieve an $\varepsilon$-expected duality gap, i.e., $\bbE[\Delta(x^{\rm out},\lambda^{\rm out})]\le \varepsilon$.

\begin{theorem}\label{thm:stoc_comp}
Let Assumption~\ref{assum:bounded_Lambda} hold.  In Algorithm~\ref{alg:SRSA}, for any starting point $(x^0,\lambda^0)\in\calX\times\Lambda$, denote $C_{\sf stoc}^\rmP$ and $C_{\sf stoc}^\rmD$ as the primal and dual oracle complexities to achieve an $\varepsilon$-expected duality gap, respectively. Then we have 
\begin{align}
C_{\sf stoc}^\rmP &= O\bigg((n+\sqrt{n\kappa_\calX})\sqrt{\frac{L_\rmD}{\varepsilon}}\log\bigg(\frac{\kappa_\calX L_\rmD(n+\sqrt{n\kappa_\calX})}{\varepsilon}\bigg)\bigg),\\
C_{\sf stoc}^\rmD &= O\bigg(\bigg(n\sqrt{\frac{L_\rmD}{\varepsilon}}+\frac{\sqrt{nL_{\lambda\lambda}L_\rmD}}{\varepsilon}\bigg)\log\bigg(\frac{L_{\lambda\lambda}(n+\sqrt{nL_{\lambda\lambda}/L_\rmD})}{\varepsilon}\bigg)\bigg).\label{eq:comp_dual_stoc}
\end{align}
\end{theorem}

\proof{} See Appendix. 

If we compare the results in Theorem~\ref{thm:stoc_comp} with those in Theorem~\ref{thm:det_comp}, in terms of the dependence of the primal oracle complexity on $n$, $\kappa_\calX$ and $\varepsilon$, the randomized framework (i.e., Algorithm~\ref{alg:SRSA}) indeed yields an improvement over the deterministic one, from $\tilO(n\sqrt{\kappa_\calX/\varepsilon})$ to $\tilO((n+\sqrt{n\kappa_\calX})/\sqrt{\varepsilon})$ (recall that $\tilO(\cdot)$ omits the log-factors in $n$, $\kappa_\calX$ and $\varepsilon$). Similarly, the dual oracle complexity has also been improved from $\tilO(n/\varepsilon)$ to $\tilO(n/\sqrt{\varepsilon}+\sqrt{n}/\varepsilon)$.

\subsection{Convergence with High Probability.}\label{sec:conv_whp}

Apart from convergence in expectation, given an error probability $\delta\in(0,1)$, we can modify the inexact solution criteria in Algorithm~\ref{alg:SRSA} (i.e.,~\eqref{eq:approx_lambda_1_stoc},~\eqref{eq:approx_x_stoc} and~\eqref{eq:approx_lambda_2_stoc}) to obtain an $\varepsilon$-duality gap w.p.\ at least $1-\delta$, i.e.,
 $\Pr\{\Delta(x^{\rm out},\lambda^{\rm out})\le \varepsilon\}\ge 1-\delta$. 

\begin{theorem}\label{thm:conv_whp}
Let Assumption~\ref{assum:bounded_Lambda} hold, $\varepsilon>0$ and $\delta\in(0,1)$ be given. In Algorithm~\ref{alg:SRSA}, choose the input parameters $\rho_{0}$, $\{\tau_k\}_{k\in\bbZ_+}$, $\{\gamma_k\}_{k\in\bbZ_+}$ and $\{\eta_k\}_{k\in\bbZ_+}$ in the same way as in Theorem~\ref{thm:main}, fix the total number of iterations $K\in\bbN$ and modify the inexact criteria~\eqref{eq:approx_lambda_1_stoc},~\eqref{eq:approx_x_stoc} and~\eqref{eq:approx_lambda_2_stoc} to 
\begin{align}
&\bbE\big[\psi_{\rho_{k}}^{\rmP}(x^{k})-S_{\rho_{k}}(x^{k},\tlambda_{\rho_k,\eta_{k}}(x^{k}))\,\big\vert\,\calF_{k,0}\big]\leq\eta_{k}\delta/(3K)\quad\mbox{a.s.},\label{eq:inexact_lambda1_hp}\\
&\bbE\big[ S(\tilx_{\gamma_{k}}(\hat{\lambda}^{k}),\hat{\lambda}^{k})-\psi^\rmD(\hat{\lambda}^{k})\,\big\vert\,\calF_{k,1}\big]\leq\gamma_{k}\delta/(3K)\quad\mbox{a.s.},\label{eq:inexact_x_hp}\\
&\bbE\big[\psi_{\rho_{k+1}}^{\rmP}(x^{k+1})-S_{\rho_{k+1}}(x^{k+1},\tlambda_{\rho_{k+1},\eta_{k}}(x^{k+1}))\,\big\vert\,\calF_{k,2}\big]\leq\eta_{k}\delta/(3K)\quad\mbox{a.s.},\label{eq:inexact_lambda2_hp}
\end{align} 
respectively. If we set $K=K_{\sf det}'\defeq 2\left\lceil\sqrt{\max\{\Delta_{\rho_0}(x^0,\lambda^0),0\}/\varepsilon}\right\rceil+1$, then 
$$
\Pr\{\Delta_{\rho_K}(x^{K},\lambda^{K})\le \varepsilon\}\ge 1-\delta. 
$$
Furthermore, if we set $K=K_{\sf det}$ as in~\eqref{eq:number_iter_det}, then 

\begin{align}
\Pr\{\Delta(x^{K},\lambda^{K})\le \varepsilon\}\ge 1-\delta.\label{eq:hp_duality_gap}
\end{align}
\end{theorem}

\proof{} See Appendix. 

Based on the inexact criteria in Theorem~\ref{thm:conv_whp}, we can also derive the primal and dual oracle complexities of obtaining an $\varepsilon$-duality gap w.p.\ at least $1-\delta$. The derivation is essentially the same as that of Theorem~\ref{thm:stoc_comp}, hence it is omitted.

\begin{theorem}\label{thm:whp_comp}
Let Assumption~\ref{assum:bounded_Lambda} hold and $\varepsilon>0$ and $\delta\in(0,1)$ be given.  
In Algorithm~\ref{alg:SRSA}, modify the inexact solution criteria~\eqref{eq:approx_lambda_1_stoc},~\eqref{eq:approx_x_stoc} and~\eqref{eq:approx_lambda_2_stoc}  in the same way as in Theorem~\ref{thm:conv_whp}. For any starting point $(x^0,\lambda^0)\in\calX\times\Lambda$, denote $C_{\sf hp}^\rmP$  and  $C_{\sf hp}^\rmD$ as the primal and dual oracle complexities to achieve an $\varepsilon$-duality gap w.p.\ at least $1-\delta$. Then  

\begin{align}
C_{\sf hp}^\rmP &= O\left((n+\sqrt{n\kappa_\calX})\sqrt{{L_\rmD}/{\varepsilon}}\log\big({\kappa_\calX L_\rmD(n+\sqrt{n\kappa_\calX})}/({\varepsilon}\delta)\big)\right),\\
C_{\sf hp}^\rmD &= O\Big(\big(n\sqrt{{L_\rmD}/{\varepsilon}}+{\sqrt{nL_{\lambda\lambda}L_\rmD}}/{\varepsilon}\big)\log\Big({L_{\lambda\lambda}L_\rmD(n+\sqrt{nL_{\lambda\lambda}/L_\rmD})}/({\varepsilon}\delta)\Big)\Big).
\end{align}
\end{theorem}
\revise{We end this section by summarizing the complexity results and providing a table that compares the complexities of our proposed method with existing methods, see Table~\ref{table:Primal_Dual_norm}. In summary,  by assuming that $\Lambda$ is bounded (cf.\ Assumption~\ref{assum:bounded_Lambda}), to reach an $\varepsilon$-duality gap (defined in~\eqref{eq:dualityGap}), the primal and dual oracle complexities are $\tilO(n\sqrt{\kappa_\calX/\varepsilon})$ and $\tilO(n/\varepsilon)$, respectively (where $\tilO(\cdot)$ hides the $\log n$ and $\log(1/\varepsilon)$ factors). 
Compared with existing works for the case $L_{\lambda\lambda}>0$ (cf.~Table~\ref{table:Primal_Dual_norm}), the primal oracle complexity of our framework is significantly better, while the dual oracle complexity is competitive. Regarding Algorithm~\ref{alg:SRSA}, by assuming the boundedness of $\Lambda$, to reach an $\varepsilon$-expected duality gap, the primal and dual oracle complexities are $\tilO((n+\sqrt{n\kappa_\calX})/\sqrt{\varepsilon})$ and $\tilO(n/\sqrt{\varepsilon}+ \sqrt{n}/\varepsilon)$, respectively, which significantly improve over those of Algorithm~\ref{alg:SMA}.  In addition, we show that  Algorithm~\ref{alg:SRSA} also converges {\em with high probability}.}

\renewcommand{\arraystretch}{1.5}
\setlength{\tabcolsep}{10pt}
{\small \begin{table}[t]\centering
\caption{Comparison of primal and dual oracle complexities with existing methods.}\label{table:Primal_Dual_norm}
\begin{threeparttable}
\begin{tabular}{|c|c|c|}\hline
 Algorithms & Primal Oracle Comp. & Dual Oracle Comp.\\\cline{1-3}
PDHG-type~\cite{Zhao_19} & $O(n/\varepsilon)$ & $O(n/\varepsilon)$\\\hline
Mirror-Prox~\cite{Nemi_05} & $O(n/\varepsilon)$ & $O(n/\varepsilon)$ \\\hline
Det.\ IPDS (Algo.~\ref{alg:SMA}) & $\tilO(n\sqrt{\kappa_\calX/\varepsilon})$ & $\tilO(n/\varepsilon)$\\\hline
Rand.\ IPDS (Algo.~\ref{alg:SRSA})\tnotex{fnt:exp_comp} & $\tilO((n+\sqrt{n\kappa_\calX})/\sqrt{\varepsilon})$ & $\tilO(n/\sqrt{\varepsilon}+ \sqrt{n}/\varepsilon)$\\\hline
\end{tabular}
\begin{tablenotes}\footnotesize
\item[1]\label{fnt:exp_comp} Both primal and dual oracle complexities of Rand.\ IPDS correspond to obtaining {\em expected} duality gap (cf.\ Theorem~\ref{thm:main_stoc}) when $L_{\lambda\lambda}>0$. 
\end{tablenotes}
\end{threeparttable}
\end{table} 
}

\section{Convex Optimization with Functional Constraints.} \label{sec:Applications_cvx_prog}

In this section, we apply  our IPDS frameworks (i.e., Algorithms~\ref{alg:SMA} and~\ref{alg:SRSA}) to  the Lagrangian (saddle point) problems associated with the constrained convex  problems.  

\subsection{Problem Setting.}\label{sec:cvx_problem}

We consider  
\begin{align}
{\min}_{x\in \calX} \;f(x)+r(x) \quad\quad\st\quad\quad &g_{i}(x) \leq0,\,\forall \,i\in[n],\label{eq:PRIMAL}
\end{align}
where \revise{$\calX\subseteq \bbR^d$} is nonempty, convex and compact, $f$ is $\mu$-s.c.\  and $L$-smooth on $\calX$, $r$ is CCP and admits a tractable BPP on $\calX$ (with DGF $\baromega$; cf.\ Section~\ref{sec:def_1}),  and for each $i\in[n]$,  $g_{i}$ is convex and $\alpha_i$-smooth on $\calX$ (where $\alpha_i\ge 0$). We assume that there exists a Slater point $\barx\in\calX^o$ (recall that $\calX^o=\calX\cap\inter\dom \baromega$) such that $g_i(\barx)<0$, for any $i\in[n]$. 
Under these conditions,~\eqref{eq:PRIMAL} has the unique primal optimal solution $x^*\in\calX$, a dual optimal solution $\lambda^*\in\bbR_+^n$ and  zero duality gap. 
Moreover, any such $(x^*,\lambda^*)$ is a saddle point of the  Lagrangian problem associated with~\eqref{eq:PRIMAL}: 
\begin{equation}
{\min}_{x\in \calX}\;{\max}_{\lambda\in\mathbb{R}_{+}^{n}}\;\big\{S(x,\lambda) \triangleq f(x)+r(x)+(1/n)\textstyle{\sum}_{i=1}^{n}n\lambda_{i}g_{i}(x)\big\},\label{eq:Lagragian}
\end{equation}
where $\lambda_i$ denotes the $i$-th entry of $\lambda$. 
In addition, any saddle point $(x^*,\lambda^*)$ of~\eqref{eq:Lagragian} is a primal-dual optimal solution pair  for~\eqref{eq:PRIMAL} with zero duality gap~\cite[Section~5.4]{Boyd_04}. This establishes the {\em equivalence} of solving~\eqref{eq:PRIMAL} and~\eqref{eq:Lagragian}.

Indeed, we observe that~\eqref{eq:Lagragian} has the same form as the SPP in~\eqref{eq:SP}. Specifically, if we set $g=r$, $h\equiv 0$, $\Lambda=\bbR_+^n$ and $\Phi_i(x,\lambda)=n\lambda_ig_i(x)$ in~\eqref{eq:SP}, then we recover~\eqref{eq:Lagragian}. As a result, $L_{xx}^i(\lambda)=n\lambda_i\alpha_i$, $L_{\lambda\lambda}^i=0$ and
\begin{align}
&L_{\lambda x}^i=nM_i,\quad\mbox{where}\quad M_i\defeq \alpha_i D_\calX + {\inf}_{x\in\calX} \normt{\nabla g_i(x)}_{*}. \label{eq:M_i}
\end{align}
In~\eqref{eq:M_i}, we recall that $D_\calX<+\infty$ denotes the diameter of the set $\calX$. 
(To obtain~\eqref{eq:M_i}, we note that $M_i\le\sup_{x\in\calX} \normt{\nabla g_i(x)}_*$. Then, by the $\alpha_i$-smoothness of $g_i$, we have $\normt{\nabla g_i(x)}_* \le \alpha_i\norm{x-x'}+\normt{\nabla g_i(x')}_*\le \alpha_iD_\calX+\normt{\nabla g_i(x')}_*$, for any $x,x'\in\calX$.) Thus,  
\begin{equation}
L_{xx}(\lambda) = \textstyle{\sum}_{i=1}^n \lambda_i\alpha_i,\quad L_{\lambda x}=M\defeq\textstyle{\sum}_{i=1}^n M_i,\quad
L_\rmD =  {M^{2}}/\mu. 
\end{equation}
We also observe that depending on the specific forms of the functions $f,r$ and $\{g_i\}_{i=1}^n$, and the relation between $d$ and $n$, 
given any $(x,\lambda)\in\calX\times\Lambda$, the cost of computing the primal gradient $\nabla_x S(x,\lambda)$ may be higher than that of computing the dual gradient $\nabla_\lambda S(x,\lambda)$. For simplicity, let $r\equiv 0$, then we have
$
\nabla_x S(x,\lambda)=\nabla f(x) + \textstyle\sum_{i=1}^n \lambda_i\nabla g_i(x),$ and 
$\nabla_\lambda S(x,\lambda)=(g_1(x),\ldots,g_n(x)).
$
Indeed, computing $\nabla_\lambda S(x,\lambda)$ simply involves evaluating the function values of $\{g_i\}_{i=1}^n$ at $x$. In contrast, computing $\nabla_x S(x,\lambda)$ involves evaluating the gradients of the functions $\{f\}\cup\{g_i\}_{i=1}^n$ at $x$. As a simple example, if $f(x) = (1/2)x^\top A x$ (where $A\succ 0$ and has no sparse structure), and $g_i(x) =\ln(1+\exp(-a_i^\top x))$, for each $i\in[n]$, then computing $\nabla_\lambda S(x,\lambda)$  involves $O(nd)$ (elementary) operations, whereas computing $\nabla_x S(x,\lambda)$ involves $O(d^2 + nd)$ operations.  Therefore, if $d\gg n$, then computing $\nabla_x S(x,\lambda)$ is much more expensive than computing $\nabla_\lambda S(x,\lambda)$.

In addition, since $\Phi(x,\cdot)$ is linear, we can choose $\bbE_2=(\bbR^n,\norm{\cdot}_2)$, where $\norm{\cdot}_2$ denotes the Euclidean norm. 
Subsequently,  the problem in~\eqref{eq:hprimalF_sm} now has closed-form solution, i.e.,
 \begin{align}
\big([g_i(x)]_+/\rho\big)_{i=1}^n = {\argmax}_{\lambda\in\bbR^n_+} \textstyle{\sum}_{i=1}^n \lambda_i g_i(x) - (\rho/2)\normt{\lambda}_2^2,\label{eq:exact_dual_sln}
\end{align} 
where $[\cdot]_+\defeq \max\{0,\cdot\}$ and we choose the DGF w.r.t.\ $(0,\bbR_+^n)$ to be  $\omega(\cdot)=(1/2)\normt{\cdot}_2^2$. 
However, note that  two of the assumptions that we made for~\eqref{eq:SP} fail to hold for~\eqref{eq:Lagragian}. First, in Assumption~\ref{assum:bounded_Lambda}, we assume that $\Lambda$ is bounded in~\eqref{eq:SP}, but it is unbounded in~\eqref{eq:Lagragian}. Second, we assume that $L_{xx}$ is a constant (w.r.t.\ $\lambda$) in~\eqref{eq:SP}, but it depends (linearly) on $\lambda$ in~\eqref{eq:Lagragian}. That said, both of these challenges can be overcome. 
 For the first challenge, recall from Remark~\ref{rmk:assump} that the boundedness of $\Lambda$ is needed for two purposes, i.e., solving the problem in~\eqref{eq:hprimalF_sm} inexactly and bounding the duality gap $\Delta$ via its smoothed counterpart $\Delta_\rho$. In the case of~\eqref{eq:Lagragian}, from~\eqref{eq:exact_dual_sln}, we see that the problem in~\eqref{eq:hprimalF_sm} can be solved {\em exactly}. In addition, in Section~\ref{sec:Conv_analysis_Cvx}, we  use a convergence criterion different from the duality gap $\Delta$. For these reasons, $\Lambda$ need not be bounded for~\eqref{eq:Lagragian}. 
For the second challenge, we propose to properly bound the growth of $L_{xx}(\lambda)$ in each iteration of our frameworks via bounding $\normt{\lambda}_\infty$, i.e., the $\ell_\infty$-norm of $\lambda$. 



\subsection{Convergence Analysis.} \label{sec:Conv_analysis_Cvx}


For the constrained  problem~\eqref{eq:PRIMAL}, instead of the duality gap, it is more common to use the optimality gap and constraint violation as the convergence criterion (e.g.,~\cite{Xu_17}). 
Specifically, 
for any $\varepsilon>0$, $\barx\in \calX$ is an $\varepsilon$-optimal and $\varepsilon$-feasible solution of~\eqref{eq:PRIMAL} if 

\begin{equation}
f(\barx)-f(x^*)\leq\varepsilon, \quad \mbox{and}\quad [g_{i}(\barx)]_+\leq\varepsilon, \;\forall\,i\in[n]. \label{eq:primal_criterion}
\end{equation}
Note that this is a primal convergence criterion. However, if we apply Algorithm~\ref{alg:SMA} or~\ref{alg:SRSA} to~\eqref{eq:Lagragian}, the established convergence results (in Theorems~\ref{thm:main},~\ref{thm:main_stoc} and~\ref{thm:conv_whp}) are all in terms of the smoothed duality gap $\Delta_\rho$. Thus, 
we need to relate $\Delta_\rho$ to the criteria in~\eqref{eq:primal_criterion}. 
Indeed, in the following lemma (see its proof in the Appendix), we show that if there exists $\barlambda\in\bbR_+^n$ such that both $\Delta_\rho(\barx,\barlambda)$ and $\rho$ are sufficiently small, then $\barx\in\calX$ satisfies~\eqref{eq:primal_criterion}.

\begin{lemma}\label{lem:CSdualgap}
Let 
$(x^*,\lambda^{*})\in\calX\times\bbR_+^n$ be a saddle point of~\eqref{eq:Lagragian}, so in particular $x^*$ is an optimal solution of~\eqref{eq:PRIMAL}. For any $\rho>0$ and $\epsilon\ge 0$, if there exist $\barx\!\in\! \calX$  and $\bar{\lambda}\!\in\!\bbR_+^n$  that  satisfy $\Delta_{\rho}(\barx,\bar{\lambda})\leq\epsilon$,  then 

\begin{align}
f(\bar{x})-f(x^{*})  \leq {\epsilon},\quad \;
[g_i(\bar{x})]_+&\le V_i(\epsilon,\rho)\defeq (\lambda_i^*+\|\lambda^*\|_2) \rho + \sqrt{2\epsilon\rho},\;\forall\,i\in[n].\label{eq:det_violation}
\end{align}
\end{lemma}

Using similar arguments, we can derive a stochastic version of Lemma~\ref{lem:CSdualgap}. 

\begin{lemma}\label{lem:CSdualgap_stoc}
Let $(x^*,\lambda^{*})\in\calX\times\bbR_+^n$ be a saddle point of~\eqref{eq:Lagragian} and $(\barx,\bar{\lambda})$ be a primal-dual pair such that $(\barx,\bar{\lambda})\in \calX\times\bbR_+^n$ a.s. 
For any $\rho>0$ and $\epsilon\ge 0$, if $(\barx,\bar{\lambda})$  satisfies $\bbE[\Delta_{\rho}(\barx,\bar{\lambda})]\leq\epsilon$,  then 

 \begin{align*}
\bbE[f(\bar{x})]-f(x^{*})  \leq {\epsilon},\quad \;
\bbE[[g_i(\bar{x})]_+]&\le V_i(\epsilon,\rho),\;\forall\,i\in[n],
\end{align*}
where $V_i(\epsilon,\rho)$ is defined in~\eqref{eq:det_violation}. 
For any $\delta\in(0,1)$, if we have  $\Pr\{\Delta_{\rho}(\barx,\bar{\lambda})\le \epsilon\}\ge 1-\delta$ (rather than $\bbE[\Delta_{\rho}(\barx,\bar{\lambda})]\leq\epsilon$),  then

 \begin{align*}
\Pr\{f(\bar{x})-f(x^{*})  \leq {\epsilon}\}\ge 1-\delta, \quad\;
\Pr\big\{[g_i(\bar{x})]_+\le V_i(\epsilon,\rho)\big\}\ge 1-\delta,\forall\,i\in[n].
\end{align*}
\end{lemma}

Based on Lemma~\ref{lem:CSdualgap} and the convergence results of Algorithm~\ref{alg:SMA} in terms of the smoothed duality gap (cf.\ Theorem~\ref{thm:main}), we can easily derive the following results. 

\begin{theorem}\label{thm:CS_det}
Let $(x^*,\lambda^{*})\in\calX\times\bbR_+^n$ be a saddle point of~\eqref{eq:Lagragian} and $\varepsilon>0$ be given. If we apply Algorithm~\ref{alg:SMA} to solve~\eqref{eq:Lagragian}, with the parameters $\rho_{0}$, $\{\tau_k\}_{k\in\bbZ_+}$, $\{\gamma_k\}_{k\in\bbZ_+}$ chosen in the same way as in Theorem~\ref{thm:main} and $\eta_k=0$ for any $k\in\bbZ_+$,  then for any starting point $(x^0,\lambda^0)\in\calX\times\Lambda$ and $K\in\bbN$,  

 \begin{align}
&\hspace{-1.5cm} f(x^{K})-f(x^{*})  \leq W_f(K,\varepsilon)\defeq \frac{2[\Delta_{\rho_0}(x^0,\lambda^0)]_+}{(K+1)(K+2)}+\frac{\varepsilon}{2},\label{eq:CS_det_f}\\
[g_{i}(x^{K})]_+ &\le W_{g_i}(K,\varepsilon)\label{eq:CS_det_g_i}\\
&\defeq \frac{16\left(\lambda_i^*+\|\lambda^*\|_2\right)L_\rmD+8\sqrt{L_\rmD[\Delta_{\rho_0}(x^{0},\lambda^0)]_+}}{(K+1)(K+2)}+\frac{4\sqrt{L_\rmD\varepsilon}}{K+1},\;\forall\,i\in[m].\nn
\end{align}
\end{theorem}

\proof{}{}
In Lemma~\ref{lem:CSdualgap}, let us take $\barx = x^K$, $\bar{\lambda}=\lambda^{K}$, $\rho=\rho_{K}$  and $\epsilon=[B'_\Delta(K,\varepsilon)]_+$ (where $B'_\Delta(K,\varepsilon)$ is defined in~\eqref{eq:Delta_k_sm_conv_rate}). Using that $[a+b]_+\le [a]_+ + [b]_+$  (for any $a,b\in\bbR$), we have $[B'_\Delta(K,\varepsilon)]_+\le W_f(K,\varepsilon)$. Thus we obtain~\eqref{eq:CS_det_f}. Using the analytic expression of $\rho_K$ in~\eqref{eq:expression_rhoK}, and $\epsilon\le W_f(K,\varepsilon)$, we also obtain~\eqref{eq:CS_det_g_i}. 
~\hfill\qed\endproof

Similarly, based on Lemma~\ref{lem:CSdualgap_stoc}, Theorems~\ref{thm:main_stoc} and~\ref{thm:conv_whp}, we can show the following results using the same reasoning that leads to Theorem~\ref{thm:CS_det}.

\begin{theorem}\label{thm:CS_stoc}
Let $(x^*,\lambda^{*})\in\calX\times\bbR_+^n$ be a saddle point of~\eqref{eq:Lagragian} and $\varepsilon>0$ be given. Let us apply Algorithm~\ref{alg:SRSA} to solve~\eqref{eq:Lagragian}, with the input parameters $\rho_{0}$, $\{\tau_k\}_{k\in\bbZ_+}$, $\{\gamma_k\}_{k\in\bbZ_+}$ and $\{\eta_k\}_{k\in\bbZ_+}$ chosen in the same way as in Theorem~\ref{thm:CS_det},  and  the starting point chosen to be any $(x^0,\lambda^0)\in\calX\times\Lambda$. Then for any $K\in\bbN$, 

\begin{align}
&\bbE[f(x^{K})]-f(x^{*})  \leq W_f(K,\varepsilon),\label{eq:CS_stoc_f}\\
&\bbE[[g_{i}(x^{K})]_+] \le W_{g_i}(K,\varepsilon),\quad\;\forall\,i\in[m]. \label{eq:CS_stoc_g_i}
\end{align}

In addition, for any $\delta\in(0,1)$, if we choose $K=K'_{\sf det}$ 
and modify the inexact solution criteria~\eqref{eq:approx_lambda_1_stoc},~\eqref{eq:approx_x_stoc} and~\eqref{eq:approx_lambda_2_stoc} in the same way as in Theorem~\ref{thm:conv_whp}, then 
 
\begin{align}
&\Pr\{f(x^{K})-f(x^{*})  \leq W_f(K,\varepsilon)\}\ge 1-\delta,\label{eq:CS_hp_f}\\
&\Pr\{[g_{i}(x^{K})]_+ \le W_{g_i}(K,\varepsilon)\}\ge 1-\delta,\quad\;\forall\,i\in[m]. \label{eq:CS_hp_g}
\end{align} 
\end{theorem}

From Theorems~\ref{thm:CS_det} and~\ref{thm:CS_stoc}, we see that for Algorithm~\ref{alg:SMA} to find an $\varepsilon$-optimal and $\varepsilon$-feasible solution of~\eqref{eq:PRIMAL}, or for Algorithm~\ref{alg:SRSA} to find such a solution in expectation (i.e., a solution that satisfies both~\eqref{eq:CS_stoc_f} and~\eqref{eq:CS_stoc_g_i}), the number of iterations needed is the same, which is denoted by $K_{\sf cons}$. Furthermore, we have %

\begin{equation}
K_{\sf cons} =O(\sqrt{L_\rmD/\varepsilon}) = O(M/\sqrt{\varepsilon\mu}).\label{eq:K_cons}
\end{equation}

\subsection{Oracle Complexity.}\label{sec:oracle_complexity_det}

From Section~\ref{sec:cvx_problem}, we notice that in Algorithms~\ref{alg:SMA} and~\ref{alg:SRSA}, the dual sub-problems  can be solved exactly. Therefore, we focus on analyzing their primal oracle complexities, where the sub-routines $\rvN_2$ and $\rvM_2$ remain the same as the ones in Sections~\ref{sec:solve_subproblem_det} and~\ref{sec:solve_subproblem_stoc}, respectively. 
Note that based on our oracle model in Section~\ref{sec:oracle_model}, in the case of~\eqref{eq:Lagragian}, the primal oracle $\scO^\rmP$ returns  $\nabla f(x)$ with input $(x,0)$ and $\lambda_i \nabla g_i(x)$ with input $(x,\lambda,i)$.

Compared to the complexity analyses in Sections~\ref{sec:oracle_comp_det},~\ref{sec:oracle_comp_stoc} and~\ref{sec:conv_whp}, the challenge here is that $L_{xx}(\lambda)$ now depends on $\lambda$. This implies that in Algorithm~\ref{alg:SMA} or~\ref{alg:SRSA}, as $\hlambda^k$ changes over iterations, $L_{xx}(\hlambda^k)$ also changes. Although this does not affect the iteration complexity of Algorithm~\ref{alg:SMA} or~\ref{alg:SRSA} (since $L_\rmD$ depends only on $L_{\lambda x}$), it does affect the oracle complexity of solving the primal sub-problem at each iteration $k$. To overcome this challenge, we propose to  bound $\normt{\hlambda^k}_\infty$ for each $k\in\bbN$ (in either the deterministic or stochastic sense).

\begin{lemma}\label{lem:scaling_hat_lambda}
In Algorithm~\ref{alg:SMA}, if we choose the input parameters $\rho_{0}$, $\{\tau_k\}_{k\in\bbZ_+}$, $\{\gamma_k\}_{k\in\bbZ_+}$ and $\{\eta_k\}_{k\in\bbZ_+}$ in the same way as in Theorem~\ref{thm:CS_det}, then for any $k\in\bbN$, 

\begin{equation}
\normt{\hat{\lambda}^k}_\infty=O(1+k\sqrt{\varepsilon\mu}/M).\label{eq:bound_hatlambda} 
\end{equation} 
\end{lemma}

\proof{}
From~\eqref{eq:exact_dual_sln}, we have that $\tlambda_{\rho_k,0}(x^k) = \big([g_i(x^k)]_+/\rho_k\big)_{i=1}^n$. By the bound on $[g_i(x^k)]_+$ in~\eqref{eq:CS_det_g_i} and the expression of $\rho_k$ in~\eqref{eq:expression_rhoK}, we have 

\begin{equation}
\normt{\tlambda_{\rho_k,0}(x^k)}_\infty=O(1+k\sqrt{\varepsilon\mu}/M). \label{eq:bound_tlambda_infty}
\end{equation}
By step~\ref{step:interp_dual2} and the choice of $\{\tau_k\}_{k\in\bbZ_+}$ in Theorem~\ref{thm:CS_det}, we have

\begin{equation}
\normt{\lambda^k}_\infty \le \frac{k}{k+2}\normt{\lambda^{k-1}}_\infty + \frac{2}{k+2}\normt{\tlambda_{\rho_k,0}(x^k)}_\infty. \label{eq:lin_recur_cvx}
\end{equation}
Based on~\eqref{eq:lin_recur_cvx}, we use Lemma~\ref{lem:lin_recursion} to conclude that 

\begin{align}
\normt{\lambda^K}_\infty &\le 
\frac{2}{(K+1)(K+2)}\left(\normt{\lambda^{0}}_\infty + {\sum}_{k=0}^K (k+1)\normt{\tlambda_{\rho_{k},0}(x^{k})}_\infty\right)\label{eq:bound_lambda_k_infty}\\
&\eqa O\left(1+K\sqrt{\varepsilon\mu}/M\right)\nn,
\end{align}
where in (a) we  use~\eqref{eq:bound_tlambda_infty}. 
Finally, by step~\ref{step:interp_dual1}, we have 

\begin{equation}
\normt{\hat{\lambda}^k}_\infty\le \tau_k\normt{\lambda^k}_\infty+(1-\tau_k)\normt{\tlambda_{\rho_k,0}(x^k)}_\infty.\label{eq:lin_recur_cvx2}
\end{equation}
We then substitute~\eqref{eq:bound_tlambda_infty} and~\eqref{eq:bound_lambda_k_infty} into~\eqref{eq:lin_recur_cvx2}, and obtain~\eqref{eq:bound_hatlambda}. 
~\hfill\qed\endproof

Based on Lemma~\ref{lem:scaling_hat_lambda}, we can bound $L_{xx}(\hlambda^k)$ via 

\begin{equation}
L_{xx}(\hlambda^k) \le \alpha\normt{\hlambda^k}_\infty = O(\alpha+k\alpha\sqrt{\varepsilon\mu}/M),\quad \mbox{where}\quad \alpha\defeq \textstyle{\sum}_{i=1}^n \alpha_i. \label{eq:bound_L_xx}
\end{equation}
Based on this bound, the oracle complexity of $\rvN_2$ for (approximately) solving the sub-problem in~\eqref{eq:hdualF} (cf.~\eqref{eq:comp_N2}), and the iteration complexity of Algorithm~\ref{alg:SMA}, i.e., $K_{\sf cons}$ in~\eqref{eq:K_cons}, we can derive the following result (see its proof in the Appendix).  

\begin{theorem}\label{thm:comp_CS}
In Algorithm~\ref{alg:SMA}, for any starting point $(x^0,\lambda^0)\in\calX\times\bbR_+$, denote $\barC_{\sf det}$ as the oracle complexity to obtain an $\varepsilon$-optimal and $\varepsilon$-feasible solution. Then 

\begin{equation}
\barC_{\sf det} =  O\bigg(\frac{nM}{\sqrt{\mu\varepsilon}}\sqrt{({L+\alpha})/{\mu}}\log\bigg(\frac{L+\alpha}{\varepsilon}\bigg)\bigg).\label{eq:comp_CS_det}
\end{equation}

\end{theorem}

Based on the oracle complexity of $\rvM_2$ for (approximately) solving the sub-problem in~\eqref{eq:hdualF} (cf.~\eqref{eq:comp_M2}), by using the same arguments as in Theorem~\ref{thm:comp_CS}, we can also derive the following oracle complexity for Algorithm~\ref{alg:SRSA}. 

\begin{theorem}\label{thm:comp_CS_stoc}
In Algorithm~\ref{alg:SRSA}, for any starting point $(x^0,\lambda^0)\in\calX\times\bbR_+$, denote $\barC_{\sf stoc}$ as the oracle complexity to obtain an $\varepsilon$-optimal and $\varepsilon$-feasible solution in expectation. Then 

\begin{equation}
\barC_{\sf stoc} =  O\bigg(\frac{\sqrt{n}M}{\sqrt{\mu\varepsilon}}\Big(\sqrt{n}+\sqrt{({L+\alpha})/{\mu}}\Big)\log\bigg(\frac{nM(L+\alpha)}{\mu\varepsilon}\bigg)\bigg).\label{eq:comp_CS_stoc}
\end{equation}
\end{theorem}

Let us compare the complexity results in Theorems~\ref{thm:comp_CS} and~\ref{thm:comp_CS_stoc}. If we interpret the factor $\kappa_{\sf cons} \defeq (L+\alpha)/\mu$ as the ``condition number'' of the constrained problem in~\eqref{eq:PRIMAL}, and recall that 
$K_{\sf cons}=O(M/\sqrt{\mu\varepsilon})$ (cf.~\eqref{eq:K_cons}), then  the oracle complexity of Algorithm~\ref{alg:SMA}, i.e., $\tilO(nK_{\sf cons}\sqrt{\kappa_{\sf cons}})$, has been reduced to $\tilO\big(\sqrt{n}K_{\sf cons}(\sqrt{n}+\sqrt{\kappa_{\sf cons}})\big)$ in Algorithm~\ref{alg:SRSA}. This is indeed consistent with our observation in Section~\ref{sec:oracle_comp_stoc}, which concerns the oracle complexities of Algorithms~\ref{alg:SMA} and~\ref{alg:SRSA} for the  SPP in~\eqref{eq:SP}.

\subsection{Related Works.}\label{sec:related_works_cvxopt}

Although numerous first-order methods have been proposed for solving the constrained problem in~\eqref{eq:PRIMAL} when $f$ is non-strongly convex , it appears that there exist few methods that tackle the case where $f$ is strongly convex (i.e., $\mu>0$). Among these,  the best-known oracle complexity is $\tilO(\varepsilon^{-1/2})$, which has been achieved by the inexact ALM method~\cite{Xu_17}, the inexact dual gradient method~\cite{Necoara_14} and the level-set method~\cite{Lin_18}. 
From Theorems~\ref{thm:comp_CS} and~\ref{thm:comp_CS_stoc}, we observe that this complexity is also achieved by Algorithms~\ref{alg:SMA} and~\ref{alg:SRSA}. 
Although all of these methods achieve the same complexity, there are two important features that distinguish our methods from the rest. First, 
 our randomized framework (i.e., Algorithm~\ref{alg:SRSA}) can effectively handle the case where the number of constraints $n$ is extremely large (cf.~Theorem~\ref{thm:comp_CS_stoc}). Second, both of our frameworks (i.e., Algorithms~\ref{alg:SMA} and~\ref{alg:SRSA}) are developed for solving the {\em general} SPP in~\eqref{eq:SP}, not only the Lagrangian problem in~\eqref{eq:Lagragian}. Therefore, they have much wider applicability compared to the other methods.

\appendix 
\section{Appendix:  Technical Proofs}

\subsection{Proof of Proposition~\ref{lem:GradLip}}

First, for any $\lambda\in\bbE_2$, we note that $\nabla_{\lambda}\hatS^\rmP(\tilx_{\gamma}(\lambda),\lambda)=\nabla_{\lambda}\Phi(\tilx_{\gamma}(\lambda),\lambda)$ and $\nabla\hpsi^\rmD(\lambda)=\nabla_{\lambda}\Phi(x^*(\lambda),\lambda)$ (cf.\ Proposition~\ref{prop:Lsmooth}). Therefore, 

\begin{align*}
&\|\nabla_{\lambda}\Phi(\tilx_{\gamma}(\lambda),\lambda)-\nabla\psi^\rmD(\lambda)\|^2_{*} = \|\nabla_{\lambda}\Phi(\tilx_{\gamma}(\lambda),\lambda)-\nabla_{\lambda}\Phi(x^*(\lambda),\lambda)\|^2_{*}\nt \\
&\lea L_{\lambda x}^2 \normt{\tilx_{\gamma}(\lambda) - x^*(\lambda)}^2\leb (2L_{\lambda x}^2/\mu) \big(\hatS^\rmP(\tilx_{\gamma}(\lambda),\lambda)-\hatS^\rmP(x^*(\lambda),\lambda)\big)\lec 2L_{\lambda x}^2\gamma/\mu, 
\end{align*}
where in (a) we use~\eqref{eq:Phi_sm(c)}, in (b) we use the $\mu$-strong convexity of $\hatS^\rmP(\cdot,\lambda)$ on $\calX$ and in (c) we use the definition of $\tilx_{\gamma}(\lambda)$ in~\eqref{eq:inexact_primal_sln}. This proves~\eqref{eq:lips_grad}.


We next prove~\eqref{eq:DescentLem}. First, for any $\lambda,\lambda'\in\bbE_2$, 

\begin{equation}
\hpsi^\rmD(\lambda')  \lea \hatS^\rmP(\tilx_{\gamma}(\lambda),\lambda') \leb \hatS^\rmP(\tilx_{\gamma}(\lambda),\lambda) + \langle\nabla_{\lambda}\hatS^\rmP(\tilx_{\gamma}(\lambda),\lambda),\lambda'-\lambda\rangle,
\end{equation}

where (a) follows from~\eqref{eq:hdualF} and (b) follows from the concavity of $\hatS^\rmP(\tilx_{\gamma}(\lambda),\cdot)$ on $\bbE_2$. This proves the left-hand side (LHS) of~\eqref{eq:DescentLem}. To show the right-hand side (RHS),  
we note that $\hpsi^\rmD$ is concave  and $L_\rmD$-smooth on $\bbE_2$ (cf.~Proposition~\ref{prop:Lsmooth}). Thus we can invoke the descent lemma~\cite{Bert_99}, such that for all $\lambda,\lambda'\in\bbE_2$, 

 {\small
\begin{align*}
\hpsi^\rmD(\lambda')&\geq\hpsi^\rmD(\lambda)+\langle\nabla\hpsi^\rmD(\lambda),\lambda'-\lambda\rangle-({L_\rmD}/{2})\|\lambda-\lambda'\|^{2}\nt\label{eq:descent_lemma}\\
& \ge \hatS^\rmP(\tilx_{\gamma}(\lambda),\lambda)-\gamma+\langle\nabla_{\lambda}\hatS^\rmP(\tilx_{\gamma}(\lambda),\lambda),\lambda'-\lambda\rangle\\
& \quad+\langle\nabla\psi^\rmD(\lambda)-\nabla_{\lambda}\hatS^\rmP(\tilx_{\gamma}(\lambda),\lambda),\lambda'-\lambda\rangle-({L_\rmD}/{2})\|\lambda-\lambda'\|^{2}\\
 & \ge \hatS^\rmP(\tilx_{\gamma}(\lambda),\lambda)-\gamma+\langle\nabla_{\lambda}\hatS^\rmP(\tilx_{\gamma}(\lambda),\lambda),\lambda'-\lambda\rangle\\
 &\quad-\|\nabla\psi^\rmD(\lambda)-\nabla_{\lambda}\hatS^\rmP(\tilx_{\gamma}(\lambda),\lambda)\|_{*}\,\|\lambda-\lambda'\|-({L_\rmD}/{2})\|\lambda-\lambda'\|^{2}\\
 & \gea \hatS^\rmP(\tilx_{\gamma}(\lambda),\lambda)-\gamma+\langle\nabla_{\lambda}\hatS^\rmP(\tilx_{\gamma}(\lambda),\lambda),\lambda'-\lambda\rangle-L_{\lambda x}\sqrt{2\gamma/\mu}\|\lambda-\lambda'\|
 -({L_\rmD}/{2})\|\lambda-\lambda'\|^2,\\
 &\geb \hatS^\rmP(\tilx_{\gamma}(\lambda),\lambda)-2\gamma+\langle\nabla_{\lambda}\hatS^\rmP(\tilx_{\gamma}(\lambda),\lambda),\lambda'-\lambda\rangle -{L_\rmD}\|\lambda-\lambda'\|^2,
\end{align*}}
where (a) follows from~\eqref{eq:lips_grad} and (b) follows from the AM-GM inequality, i.e., 

\begin{equation}
L_{\lambda x}\sqrt{2\gamma/\mu}\|\lambda-\lambda'\|\leq (L_{\lambda x}^2/\mu)\|\lambda-\lambda'\|^{2}+\gamma\le (L_\rmD/2)\|\lambda-\lambda'\|^{2}+\gamma.
\end{equation} 

We then rearrange~\eqref{eq:descent_lemma} to obtain the RHS of~\eqref{eq:DescentLem}. 

\subsection{Proof of Lemma~\ref{lem:SmDuGap}}

Fix any $k\in\bbZ_+$. 
Since $\hatS^\rmD_{\rho_{k}}(x^{k},\cdot)$ is $\rho_{k}$-strongly concave on $\Lambda$, 

\begin{align} \label{eq:strongly_concave}
\begin{array}{ll}
 \hpsi^\rmP_{\rho_{k}}(x^k)-\hatS^\rmD_{\rho_{k}}(x^{k},\lambda)&=\hatS^\rmD_{\rho_{k}}(x^{k},\lambda^*_{\rho_{k}}(x^{k}))-\hatS^\rmD_{\rho_{k}}(x^{k},\lambda)\revise{\geq} 
 \frac{\rho_{k}}{2}\|\lambda-\lambda^*_{\rho_{k}}(x^{k})\|^{2}.
 \end{array}
\end{align} 

As a result, for all $\lambda\in\Lambda$, we have 
 
\begin{align*}
\begin{array}{ll}
\Delta_{\rho_{k}}(x^k,\lambda^k)\nt\label{eq:sm_duality_gap1}
 & =\psi_{\rho_{k}}^{\rmP}(x^{k})-\psi^\rmD(\lambda^{k}) = f(x^k)+g(x^k)+\hpsi^\rmP_{\rho_{k}}(x^k)-\psi^\rmD(\lambda^{k})\\
 & \gea f(x^k)+g(x^k)+ \hatS^\rmD_{\rho_{k}}(x^{k},\lambda)+\frac{\rho_{k}}{2}\|\lambda-\lambda^*_{\rho_{k}}(x^{k})\|^{2} -\psi^\rmD(\lambda^{k})\\
 & = S(x^{k},\lambda) - \rho_k\omega(\lambda) +\frac{\rho_{k}}{2}\|\lambda-\lambda^*_{\rho_{k}}(x^{k})\|^{2} -(\hpsi^\rmD(\lambda^{k}) - h(\lambda^k))\\
 & =  \hatS^\rmP(x^{k},\lambda)- h(\lambda)-\rho_k\omega(\lambda) +\frac{\rho_{k}}{2}\|\lambda-\lambda^*_{\rho_{k}}(x^{k})\|^{2} -\hpsi^\rmD(\lambda^{k})+h(\lambda^k),
 \end{array}
\end{align*} 
where in (a) we use~\eqref{eq:strongly_concave}. Define $z_k(\lambda)\defeq\tau_{k}\lambda^{k}+(1-\tau_{k})\lambda$. We then multiply both sides of~\eqref{eq:sm_duality_gap1} by $\tau_k>0$, and obtain 
 
\begin{align*}
&\tau_{k}\Delta_{\rho_{k}}(x^k,\lambda^k) 
\gea \tau_k \hatS^\rmP(x^{k},\lambda) - \tau_k h(\lambda) - \rho_{k+1}\omega(\lambda) +\frac{\rho_{k+1}}{2}\|\lambda-\lambda^*_{\rho_{k}}(x^{k})\|^{2} -\tau_k\hpsi^\rmD(\lambda^{k})+\tau_k h(\lambda^k) \nn\\
& \quad=\tau_k \hatS^\rmP(x^{k},\lambda) + (1-\tau_{k})\hatS^\rmP(\tilx_{\gamma_{k}}(\hat{\lambda}^{k}),\lambda)  - \tau_k h(\lambda) - \rho_{k+1}\omega(\lambda) \nt\label{eq:lower_PDGap}\\
&\quad\quad+\frac{\rho_{k+1}}{2}\|\lambda-\lambda^*_{\rho_{k}}(x^{k})\|^{2}-\tau_k\hpsi^\rmD(\lambda^{k}) - (1-\tau_{k})\hatS^\rmP(\tilx_{\gamma_{k}}(\hat{\lambda}^{k}),\lambda)+\tau_k h(\lambda^k)\nn\\
&\quad \geb \hatS^\rmP(x^{k+1},\lambda) - \tau_k h(\lambda) - \rho_{k+1}\omega(\lambda)+\frac{\rho_{k+1}}{2}\|\lambda-\lambda^*_{\rho_{k}}(x^{k})\|^{2} \\
&\quad\quad-\tau_k\big(\hatS^\rmP(\tilx_{\gamma_{k}}(\hat{\lambda}^{k}),\hat{\lambda}^{k})+\langle\nabla_{\lambda}\hatS^\rmP(\tilx_{\gamma_{k}}(\hat{\lambda}^{k}),\hat{\lambda}^{k}),\,\lambda^{k}-\hat{\lambda}^{k}\rangle\big)\\
& \quad\quad- (1-\tau_{k})\big(\hatS^\rmP(\tilx_{\gamma_{k}}(\hat{\lambda}^{k}),\hlambda^k) + \langle\nabla_{\lambda}\hatS^\rmP(\tilx_{\gamma_{k}}(\hat{\lambda}^{k}),\hat{\lambda}^{k}),\,\lambda-\hat{\lambda}^{k}\rangle\big)+\tau_k h(\lambda^k)\\
 &\quad \eqc S_{\rho_{k+1}}(x^{k+1},\lambda) +(1- \tau_k) h(\lambda) -\hatS^\rmP(\tilx_{\gamma_{k}}(\hat{\lambda}^{k}),\hat{\lambda}^{k})\\
 & \quad\quad+\frac{\rho_{k+1}}{2}\|\lambda-\lambda^*_{\rho_{k}}(x^{k})\|^{2}- \langle\nabla_{\lambda}\hatS^\rmP(\tilx_{\gamma_{k}}(\hat{\lambda}^{k}),\hat{\lambda}^{k}),\,z_k(\lambda)-\hat{\lambda}^{k}\rangle+\tau_k h(\lambda^k)\\
 &\quad\ged S_{\rho_{k+1}}(x^{k+1},\lambda)  +\frac{\rho_{k+1}}{2}\|\lambda-\lambda^*_{\rho_{k}}(x^{k})\|^{2} -\big(\hpsi^\rmD(z_k(\lambda)) + L_\rmD\normt{\hlambda^k-z_k(\lambda)}^2+2\gamma_k\big)+h(z_k(\lambda)),
\end{align*}
where we use $\tau_k\rho_{k}=\rho_{k+1}$  in (a), the convexity of $\hatS^\rmP(\cdot,\lambda)$, the LHS of~\eqref{eq:DescentLem} and the concavity of $\hatS^\rmP(\tilx_{\gamma_{k}}(\hat{\lambda}^{k}),\cdot)$ in (b),  the definition of $\hatS^\rmP$ and $S_{\rho_{k+1}}$ (in~\eqref{eq:hdualF} and~\eqref{eq:regularized_saddle_function}, respectively) in (c), and the RHS of~\eqref{eq:DescentLem} and the convexity of $h$ in (d). Note that if we take $\lambda=\tlambda_{\rho_{k+1},\eta_{k}}(x^{k+1})$, then $z_k(\lambda)=\lambda^{k+1}$ by step~\ref{step:interp_dual2}. In addition, from steps~\ref{step:interp_dual1} and~\ref{step:interp_dual2}, we have

 \begin{equation}
\hlambda^k-\lambda^{k+1} = (1-\tau_k)\big(\tlambda_{\rho_k,\eta_{k}}(x^{k})- \tlambda_{\rho_{k+1},\eta_{k}}(x^{k+1})\big),\label{eq:diff_lambda}
\end{equation}
and from~\eqref{eq:approx_lambda_1} and~\eqref{eq:strongly_concave}, we have

\begin{equation}
\frac{\rho_{k}}{2}\|\tlambda_{\rho_k,\eta_{k}}(x^{k})-\lambda^*_{\rho_{k}}(x^{k})\|^{2} \le \hpsi^\rmP_{\rho_{k}}(x^k)-\hatS^\rmD_{\rho_{k}}(x^{k},\tlambda_{\rho_k,\eta_{k}}(x^{k}))\le \eta_k. \label{eq:upper_bound_sqNorm}
\end{equation}
This observation leads us to bound $\|\tlambda_{\rho_{k+1},\eta_{k}}(x^{k+1})-\lambda^*_{\rho_{k}}(x^{k})\|^{2}$ as 

\begin{align}
&\|\tlambda_{\rho_{k+1},\eta_{k}}(x^{k+1})-\lambda^*_{\rho_{k}}(x^{k})\|^{2}\label{eq:lambda_diff_norm} 
\\& \gea\frac{1}{2}\|\tlambda_{\rho_{k+1},\eta_{k}}(x^{k+1})-\tlambda_{\rho_k,\eta_{k}}(x^{k})\|^{2}-\|\tlambda_{\rho_k,\eta_{k}}(x^{k})-\lambda^*_{\rho_{k}}(x^{k})\|^{2} \nn\\
 & \geb\frac{1}{2}\|\tlambda_{\rho_{k+1},\eta_{k}}(x^{k+1})-\tlambda_{\rho_k,\eta_{k}}(x^{k})\|^{2}-\frac{2\eta_{k}}{\rho_{k}}\nn\eqc \frac{\normt{\lambda^{k+1}-\hat{\lambda}^k}^2}{2(1-\tau_k)^2}-\frac{2\eta_{k}}{\rho_{k}},\nn
\end{align}
where in (a) we use $\norm{a+b}^2\le 2(\norm{a}^2+\norm{b}^2)$, in (b) we use~\eqref{eq:upper_bound_sqNorm} and in (c) we use~\eqref{eq:diff_lambda}.  We then substitute $\lambda=\tlambda_{\rho_{k+1},\eta_{k}}(x^{k+1})$ and~\eqref{eq:lambda_diff_norm} into~\eqref{eq:lower_PDGap}, and obtain

\begin{align*}
\begin{array}{ll}
\tau_{k}\Delta_{\rho_{k}}(x^k,\lambda^k)& \ge S_{\rho_{k+1}}(x^{k+1},\tlambda_{\rho_{k+1},\eta_{k}}(x^{k+1}))  +\frac{\rho_{k+1}}{2}\bigg(\frac{\normt{\lambda^{k+1}-\hat{\lambda}^k}^2}{2(1-\tau_k)^2}-\frac{2\eta_{k}}{\rho_{k}}\bigg)\\
 & \qquad-\big(\psi^\rmD(\lambda^{k+1}) + L_\rmD\normt{\hlambda^k-\lambda^{k+1}}^2+2\gamma_k\big)\\
& \gea  \psi^\rmP_{\rho_{k+1}}(x^{k+1}) - (1+\tau_k)\eta_k +\bigg(\frac{\rho_{k+1}}{4(1-\tau_k)^2} - L_\rmD\bigg)\normt{\lambda^{k+1}-\hat{\lambda}^k}^2-\psi^\rmD(\lambda^{k+1}) - 2\gamma_k\\
 &\geb  \Delta_{\rho_{k+1}}(x^{k+1},\lambda^{k+1}) - 2\eta_k-2\gamma_k,
 \end{array}
\end{align*} 
where we use~\eqref{eq:approx_lambda_2}  in (a)  and $\tau_k\in(0,1)$ and ${\rho_{k+1}}\geq {4(1-\tau_{k})^{2}}L_\rmD$ in (b). 

\subsection{Proof of Theorem~\ref{thm:det_comp}}

 Since $\gamma_k=\varepsilon/(4(k+3))=O(\varepsilon/k)$ (cf.~\eqref{eq:choose_param_SMA}), based on~\eqref{eq:comp_N2}, we have

\begin{align}
C_{\sf det}^\rmP &= {\sum}_{k=1}^{K_{\sf det}} \; O\left(n\sqrt{\kappa_\calX}\log\big((L+L_{xx})k/\varepsilon\big)\right)\\
&= O\left(n\sqrt{\kappa_\calX}\Big(K_{\sf det}\log\big(\kappa_\calX\big)+\log\big(K_{\sf det}!\big)\Big)\right) \nn\\
&\eqa O\left(n\sqrt{\kappa_\calX}\sqrt{L_\rmD/\varepsilon}\Big(\log\big((L+L_{xx})/\varepsilon\big)+\log\big({L_\rmD/\varepsilon}\big)\Big)\right)\nn= O\left(n\sqrt{\kappa_\calX L_\rmD/\varepsilon}\log\big((L+L_{xx})L_\rmD/\varepsilon\big)\right),\nn
\end{align}
where in (a) we use the fact that $\log(K!)\!=\!\Theta(K\log K)$ for any $K\in\bbN$ and~\eqref{eq:number_iter_det}. 

Similarly, we can analyze the dual oracle complexity for solving~\eqref{eq:approx_lambda_1}.  Since $\rho_k=O(L_\rmD/k^2)$ (cf.~\eqref{eq:expression_rhoK}) and $\eta_k=O(\varepsilon/k)$ (cf.~\eqref{eq:choose_param_SMA}), based on~\eqref{eq:comp_N1}, we have

{\small
\begin{align}
C_{{\sf det},1}^\rmD &= {\sum}_{k=1}^{K_{\sf det}} \; O\Big(n\sqrt{L_{\lambda\lambda}k^2/L_\rmD}\log(L_{\lambda\lambda}k/\varepsilon)\Big)\label{eq:C_det_D}\\
&= O\Big(n\sqrt{L_{\lambda\lambda}/L_\rmD} \Big(\log(L_{\lambda\lambda}/\varepsilon){\sum}_{k=1}^{K_{\sf det}}k + {\sum}_{k=1}^{K_{\sf det}}\; k\log k\Big)\Big)\nn\\
&\eqa O\Big(n\sqrt{L_{\lambda\lambda}/L_\rmD} ({L_\rmD/\varepsilon})\Big(\log(L_{\lambda\lambda}/\varepsilon) +  \log({L_\rmD/\varepsilon}) \Big)\Big)\nn 
= O\Big(n\sqrt{L_{\lambda\lambda}L_\rmD}/\varepsilon\log(L_{\lambda\lambda}L_\rmD/\varepsilon)\Big)\nn,
\end{align}}
where in (a) we use $\sum_{k=1}^K k =\Theta(K^2)$, $\sum_{k=1}^K k\log k =\Theta(K^2\log K)$ and~\eqref{eq:number_iter_det}. We can  repeat this analysis to conclude that the dual oracle complexity for solving~\eqref{eq:approx_lambda_2}, i.e., $C_{{\sf det},2}^\rmD$ has the same order as $C_{{\sf det},1}^\rmD$. Since $C_{\sf det}^\rmD=C_{{\sf det},1}^\rmD + C_{{\sf det},2}^\rmD$, the proof is complete. 

\subsection{Proof of Lemma~\ref{lem:SmDuGap_stoc}}

To prove this lemma, one needs to properly incorporate the inexact criteria in~\eqref{eq:approx_lambda_1_stoc},~\eqref{eq:approx_x_stoc} and~\eqref{eq:approx_lambda_2_stoc} (which involve conditional expectations) into the proof of Lemma~\ref{lem:SmDuGap}. The key steps are: i)
taking conditional expectation over the steps in the proof of Lemma~\ref{lem:SmDuGap} by using the measurability results in~\eqref{eq:measurable_RV} and ii) applying the tower property of conditional expectation by using the nested relation in~\eqref{eq:nested}. 
Specifically, at the $k$-th iteration, we first 
modify the proof of Proposition~\ref{prop:Lsmooth} and  show that 

\begin{align*}
\bbE[\hatS^\rmP(\tilx_{\gamma}(\hlambda^k),\hlambda^k)&-\hpsi^\rmD(\lambda^{k+1})+\big\langle \nabla_{\lambda}\hatS^\rmP(\tilx_{\gamma}(\hlambda^k),\hlambda^k),\lambda^{k+1}-\hlambda^k\big\rangle\,|\,\calF_{k,1}]\nt\label{eq:Descent_lemmaExp}\\
& \leq L_\rmD\bbE[\|\hlambda^k-\lambda^{k+1}\|^2\,|\,\calF_{k,1}]+2\gamma_k.
\end{align*}
(For notational brevity, we omit `a.s.' here and for all the inequalities below.)
Furthermore, since $\calF_{k,0}\subseteq\calF_{k,1}$, if we take $\bbE[\cdot|\calF_{k,0}]$ over~\eqref{eq:Descent_lemmaExp}, then we have

\begin{align*}
\bbE[\hatS^\rmP(\tilx_{\gamma}(\hlambda^k),\hlambda^k)&-\hpsi^\rmD(\lambda^{k+1})+\big\langle \nabla_{\lambda}\hatS^\rmP(\tilx_{\gamma}(\hlambda^k),\hlambda^k),\lambda^{k+1}-\hlambda^k\big\rangle\,|\,\calF_{k,0}]\nt\label{eq:Descent_lemmaExp2}\\
& \leq L_\rmD\bbE[\|\hlambda^k-\lambda^{k+1}\|^2\,|\,\calF_{k,0}]+2\gamma_k. 
\end{align*}
In addition, from~\eqref{eq:lambda_diff_norm}, we have 

\begin{align}
\bbE[\|\tlambda_{\rho_{k+1},\eta_{k}}(x^{k+1})-\lambda^*_{\rho_{k}}(x^{k})\|^{2}\,|\,\calF_{k,0}]\ge \frac{\bbE[\normt{\lambda^{k+1}-\hat{\lambda}^k}^2\,|\,\calF_{k,0}]}{2(1-\tau_k)^2}-\frac{2\eta_{k}}{\rho_{k}}. \label{eq:diffNormLambdaExp}
\end{align}
Now, we can take $\bbE[\cdot|\calF_{k,0}]$ over Equation~(c) in~\eqref{eq:lower_PDGap}, and use~\eqref{eq:Descent_lemmaExp2},~\eqref{eq:diffNormLambdaExp} and the fact that $x^k,\lambda^k\in\calF_{k,0}$ to obtain 

\begin{align}
\tau_{k}\Delta_{\rho_{k}}(x^k,\lambda^k) 
\ge & \bbE[S_{\rho_{k+1}}(x^{k+1},\tlambda_{\rho_{k+1},\eta_{k}}(x^{k+1}))  +\frac{\rho_{k+1}\normt{\lambda^{k+1}-\hat{\lambda}^k}^2}{4(1-\tau_k)^2}-\tau_k\eta_{k}\label{eq:penultimate_exp}\\
 & -\big(\psi^\rmD(\lambda^{k+1}) + L_\rmD\normt{\hlambda^k-\lambda^{k+1}}^2+2\gamma_k\big)\,|\,\calF_{k,0}].\nn
\end{align}
Again, since $\calF_{k,0}\subseteq\calF_{k,2}$, if we take $\bbE[\cdot|\calF_{k,0}]$ over~\eqref{eq:approx_lambda_2_stoc}, then we have

\begin{align}
\bbE\big[\psi_{\rho_{k+1}}^{\rmP}(x^{k+1})-S_{\rho_{k+1}}(x^{k+1},\tlambda_{\rho_{k+1},\eta_{k}}(x^{k+1}))\,\big\vert\,\calF_{k,0}\big]\leq\eta_{k}.\label{eq:inexaxt3}
\end{align}
We then substitute~\eqref{eq:inexaxt3} into~\eqref{eq:penultimate_exp}, and use the condition ${\rho_{k+1}}\geq {4(1-\tau_{k})^{2}}L_\rmD$ to get

\begin{align}
\tau_{k}\Delta_{\rho_{k}}(x^k,\lambda^k) \ge \bbE[\Delta_{\rho_{k+1}}(x^{k+1},\lambda^{k+1})\,|\,\calF_{k,0}] - 2\eta_k-2\gamma_k.
\end{align}

\subsection{Proof of Theorem~\ref{thm:stoc_comp}}

 The proof follows the same argument as that of Theorem~\ref{thm:det_comp}, hence we only outline the important steps. 
Based on the choice of $\gamma_k$ in~\eqref{eq:choose_param_SMA} and the complexity of $\rvM_2$ in~\eqref{eq:comp_M2}, 

\begin{align}
C_{\sf stoc}^\rmP &= {\sum}_{k=1}^{K_{\sf stoc}} \; O\big((n+\sqrt{n\kappa_\calX})\log\big((L+L_{xx})(n+\sqrt{n\kappa_\calX})k/(\mu\varepsilon)\big)\big)\nn\\
&= O\left((n+\sqrt{n\kappa_\calX})\Big(K_{\sf stoc}\log\big((L+L_{xx})(n+\sqrt{n\kappa_\calX})/(\mu\varepsilon)\big)+\log\big(K_{\sf stoc}!\big)\Big)\right)\nn\\
&= O\left((n+\sqrt{n\kappa_\calX})\sqrt{L_\rmD/\varepsilon}\Big(\log\big((L+L_{xx})(n+\sqrt{n\kappa_\calX})/(\mu\varepsilon)\big)+\log\big({L_\rmD/\varepsilon}\big)\Big)\right)\nn\\
&= O\big((n+\sqrt{n\kappa_\calX})\sqrt{L_\rmD/\varepsilon}\log\big((L+L_{xx})L_\rmD(n+\sqrt{n\kappa_\calX})/(\mu\varepsilon)\big)\big).\nn
\end{align}
Using the same reasoning as in the proof of Theorem~\ref{thm:det_comp}, the dual oracle complexities for solving both~\eqref{eq:approx_lambda_1_stoc} and~\eqref{eq:approx_lambda_2_stoc} have the same order, so it suffices to only analyze the complexity for solving~\eqref{eq:approx_lambda_1_stoc}. Specifically, based on~\eqref{eq:comp_M1}, we have

{\small\begin{align}
C_{{\sf stoc},1}^\rmD &= \textstyle{\sum}_{k=1}^{K_{\sf stoc}} \; O\Big(\big(n+\sqrt{nL_{\lambda\lambda}k^2/L_\rmD}\big)\log\big(L_{\lambda\lambda}(n+\sqrt{nL_{\lambda\lambda}k^2/L_\rmD})k/(L_\rmD\varepsilon)\big)\Big)\nn\\
&\eqa \textstyle{\sum}_{k=1}^{K_{\sf stoc}} \; O\Big(\big(n+k\sqrt{nL_{\lambda\lambda}/L_\rmD}\big)\big(\log\big(L_{\lambda\lambda}(n+\sqrt{nL_{\lambda\lambda}/L_\rmD})/(L_\rmD\varepsilon)\big)+\log k\big)\Big)\nn\\
&= O\Big(\big(K_{\sf stoc}n+\sqrt{nL_{\lambda\lambda}/L_\rmD}\textstyle{\sum}_{k=1}^{K_{\sf stoc}}k\big)\log\big(L_{\lambda\lambda}(n+\sqrt{nL_{\lambda\lambda}/L_\rmD})/(L_\rmD\varepsilon)\big)\nn\\
&\hspace{5cm}+\big(n\textstyle{\sum}_{k=1}^{K_{\sf stoc}}\log k+\sqrt{L_{\lambda\lambda}/L_\rmD}\textstyle{\sum}_{k=1}^{K_{\sf stoc}}k\log k\big)\Big)\nn\\
&= O\Big(\big(n\sqrt{L_\rmD/\varepsilon}+\sqrt{nL_{\lambda\lambda}L_\rmD}/\varepsilon\big)\log\big(L_{\lambda\lambda}(n+\sqrt{nL_{\lambda\lambda}/L_\rmD})/(L_\rmD\varepsilon)\big)\nn\\
&\hspace{5cm}+\sqrt{L_\rmD/\varepsilon}\log({L_\rmD/\varepsilon})(n+\sqrt{nL_{\lambda\lambda}/\varepsilon})\Big)\nn\\
& = O\Big(\big(n\sqrt{L_\rmD/\varepsilon}+\sqrt{nL_{\lambda\lambda}L_\rmD}/\varepsilon\big)\log\big(L_{\lambda\lambda}(n+\sqrt{nL_{\lambda\lambda}/L_\rmD})/\varepsilon\big)\Big)\nn,
\end{align}}
where (a) holds since $n\le kn$. We obtain~\eqref{eq:comp_dual_stoc} by noting that $C_{{\sf stoc}}^\rmD=\Theta\big(C_{{\sf stoc},1}^\rmD\big)$. 

\subsection{Proof of Theorem \ref{thm:conv_whp}}
 
First, let us define the events $\calA_{0,0}\defeq \Omega$, and for any $k\in\bbZ_+$,   

\begin{align*}
\calA_{k,1}&\defeq \{\psi_{\rho_{k}}^{\rmP}(x^{k})-S_{\rho_{k}}(x^{k},\tlambda_{\rho_k,\eta_{k}}(x^{k}))\le\eta_k\},\,
\calA_{k,2}\defeq \{S(\tilx_{\gamma_{k}}(\hat{\lambda}^{k}),\hat{\lambda}^{k})-\psi^\rmD(\hat{\lambda}^{k})\le\gamma_k\},\\
\calA_{k+1,0}&\defeq \{\psi_{\rho_{k+1}}^{\rmP}(x^{k+1})-S_{\rho_{k+1}}(x^{k+1},\tlambda_{\rho_{k+1},\eta_{k}}(x^{k+1}))\le\eta_k\}.
\end{align*}
Also, for any measurable event $\calA$, denote its complement as $\calA^\rmc$ and its indicator function as $\bbI_{\calA}$, i.e., $\bbI_{\calA}(z)=1$ if $z\in\calA$ and $0$ otherwise.

Fix any $k\in\{0,\ldots,K-1\}$. 
From Markov's inequality and~\eqref{eq:inexact_lambda1_hp}, we have

\begin{align}
&\Pr\big\{\calA^\rmc_{k,1}\,\big\vert\,\calF_{k,0}\big\}\label{eq:condition_prob}
\le   \bbE\big[\psi_{\rho_{k}}^{\rmP}(x^{k})-S_{\rho_{k}}(x^{k},\tlambda_{\rho_k,\eta_{k}}(x^{k}))\,\big\vert\,\calF_{k,0}\big]/\eta_k\le \delta/(3K) \;\;\mbox{a.s.} 
\end{align}
Since $\bigcup_{i=0}^{k-1}\{\calA_{i,1},\calA_{i,2},\calA_{i+1,0}\}\subseteq\calF_{k,0}$, 
we have that 

\begin{align}
{\calC_{k,0}}\in\calF_{k,0},\quad\mbox{where} \quad \calC_{k,0}\defeq \textstyle{\bigcap}_{i=0}^{k-1}\;\big(\calA_{i,1}\cap\calA_{i,2}\cap\calA_{i+1,0}\big). 
\end{align}
(When $k=0$, define $\calC_{0,0}\defeq \calA_{0,0}$.)  In addition, note that $\Pr\{\calC_{k,0}\}>0$, since 

\begin{align}
\Pr\big\{\calC_{k,0}^\rmc\big\} = \Pr\big\{\textstyle{\bigcup}_{i=0}^{k-1}\big(\calA_{k-1,1}^\rmc\cup\calA_{k-1,2}^\rmc\cup\calA_{k,0}^\rmc\big)\big\}\le (3k)\delta/(3K) \le \delta <1. \label{eq:positive_prob}
\end{align}
We then take conditional expectation $\bbE[\cdot\,|\,{\calC_{k,0}}]$ in~\eqref{eq:condition_prob} to obtain 

\begin{equation}
\bbE\big[\Pr\big\{\calA_{k,1}\,\big\vert\,\calF_{k,0}\big\}\,|\,{\calC_{k,0}}\big]\ge 1-\delta/(3K). \label{eq:hp_1}
\end{equation}
On the other hand, 

\begin{align*}
\bbE\big[\Pr\big\{\calA_{k,1}\,\big\vert\,\calF_{k,0}\big\}\,|\,{\calC_{k,0}}\big] &= \bbE\big[\Pr\big\{\calA_{k,1}\,\big\vert\,\calF_{k,0}\big\}\bbI_{\calC_{k,0}}\big]/\bbP(\calC_{k,0})\eqa \bbE\big[\bbI_{\calA_{k,1}}\bbI_{\calC_{k,0}}\big]/\bbP(\calC_{k,0}) = \Pr\big\{\calA_{k,1}\,|\,\calC_{k,0}\big\},
\end{align*}
where (a) follows since ${\calC_{k,0}}\in\calF_{k,0}$. Therefore, we have

 \begin{equation}
\Pr\{\calA_{k,1}\,|\,\calC_{k,0}\}\ge 1-\delta/(3K). \label{eq:hp_set_1}
\end{equation}

Similarly, if we define $\calC_{k,1}\defeq \calC_{k,0}\cap\calA_{k,1}$ and $\calC_{k,2}\defeq \calC_{k,1}\cap\calA_{k,2}$, then we also have 

\begin{align}
&\Pr\{\calA_{k,2}\,|\,{\calC_{k,1}}\}\geq 1-\delta/(3K),\quad\Pr\{\calA_{k+1,0}\,|\,{\calC_{k,2}}\}\geq 1-\delta/(3K).\label{eq:hp_set_2} 
\end{align}

From Theorem~\ref{thm:main}, we know that if $K=K_{\sf det}'$ and the event $\bigcap_{k=0}^{K-1}\big(\calA_{k,1}\cap\calA_{k,2}\cap\calA_{k+1,0}\big)$ occurs, then $\Delta_{\rho_K}(x^K,\lambda^K)\le \varepsilon$. Therefore,

\begin{align*}
\Pr\big\{\Delta_{\rho_K}(x^K,\lambda^K)\le \varepsilon\big\} &\ge \Pr\big\{\textstyle{\bigcap}_{k=0}^{K-1}\big(\calA_{k,1}\cap\calA_{k,2}\cap\calA_{k+1,0}\big)\big\}\\
&= \textstyle{\prod}_{k=0}^{K-1} \Pr\big\{\calA_{k+1,0}\,|\,{\calC_{k,2}}\big\}\Pr\big\{\calA_{k,2}\,|\,{\calC_{k,1}}\big\}\Pr\{\calA_{k,1}\,|\,{\calC_{k,0}}\}\\
&\gea \big(1-\delta/(3K)\big)^{3K}\geb 1-\delta,
\end{align*}

where (a) follows from~\eqref{eq:hp_set_1} and~\eqref{eq:hp_set_2} and (b) follows from Bernoulli's inequality. By the same reasoning, we can also show that if $K=K_{\sf det}$, then~\eqref{eq:hp_duality_gap} holds.

\subsection{Proof of Lemma~\ref{lem:CSdualgap} }

Since $\lambda^*\in\bbR_+^n$ is an optimal solution of $\max_{\lambda\in\bbR_+^n}\psi^\rmD(\lambda)$, we have that $\psi^\rmD(\barlambda)\le \psi^\rmD(\lambda^*)$. 
This implies $\Delta_{\rho}(\barx,{\lambda^*})\le\Delta_{\rho}(\barx,\barlambda) \leq\epsilon$. 
From the definition of $\Delta_\rho$ in~\eqref{eq:smoothed_duality_gap}, we have 

\begin{align}
\epsilon\ge \Delta_{\rho}(\barx,{\lambda^*})
&\ge S(\bar{x},\lambda)-({\rho}/{2})\|\lambda\|_{2}^{2}-S(x,{\lambda^*}), \quad\forall\,x\in \calX, \;\forall\,\lambda\in\mathbb{R}_{+}^{m}.\label{eq:GapCS-1}
\end{align}
We then choose $x\!=\!x^{*}$ and $\lambda\!=\!0$ in~\eqref{eq:GapCS-1} to obtain 

 \begin{align}
\epsilon \ge S(\bar{x},0)-S(x^{*},{\lambda^*}) =f(\bar{x})-\left(f(x^{*})+\textstyle{\sum}_{i=1}^n\lambda^*_i g_i(x^*)\right)   \geq f(\bar{x})-f(x^{*}), 
\end{align}
where the last step follows from $\lambda^*_i\ge 0$ and $g_i(x^*)\le 0$, for any $i\in[n]$. 

Now fix any $\theta>0$ and $i\in[n]$. Let $e_i\in\bbR^{n}$ denote the $i$-th standard basis vector, i.e., $(e_i)_i=1$ and $(e_i)_j=0$ for any $j\!\in \![n]\setminus\{i\}$. 
In~(\ref{eq:GapCS-1}), if we choose $x=x^{*}$ and $\lambda=\lambda^*+\theta_i e_i$, where $\theta_i=\theta$ if $g_i(\barx)>0$ and $0$ otherwise, then 
 \begin{align*}
\begin{array}{ll}
\epsilon & \geq S(\bar{x},\lambda^*)+\theta_i g_{i}(\bar{x})-({\rho}/{2})\|\lambda^*+\theta_i e_i\|_{2}^{2}-S(x^*,{\lambda^*}) \\
&  \geq\theta_i g_{i}(\bar{x})-({\rho}/{2})\|\lambda^*+\theta_i e_i\|_{2}^{2}\ge \theta[g_{i}(\bar{x})]_+-({\rho}/{2})\|\lambda^*+\theta e_i\|_{2}^{2},
\end{array}
\end{align*}
where in the last step we use $\lambda^*\ge 0$ and $\theta\ge \theta_i\ge 0$. 
After rearranging, we have 

 \begin{equation} \label{eq:InCS}
\begin{split}
[g_i(\bar{x})]_+&\leq \rho\lambda_i^*+{\rho\theta}/{2}+({\rho\|\lambda^*\|^2_2+2\epsilon})/({2\theta})\lea \rho\lambda_i^* + \sqrt{\rho({\rho\|\lambda^*\|^2_2+2\epsilon})}\\
&\leb(\lambda_i^*+\|\lambda^*\|_2) \rho + \sqrt{2\epsilon\rho},
 \end{split}
\end{equation}
where  we take the infimum over $\theta>0$ in (a) and use $\sqrt{a+b}\le \sqrt{a}+\sqrt{b}$, $\forall\, a,b\ge 0$  in (b). 

\subsection{Proof of Theorem~\ref{thm:comp_CS}}

Similar to the analysis in Section~\ref{sec:oracle_comp_det}, we have 

\begin{align*}
\begin{array}{ll}
\barC_{\sf det}&\eqa O\Big(n{\sum}_{k=1}^{K_{\sf cons}}\sqrt{(L+\alpha)/\mu+k\alpha\sqrt{\varepsilon}/(M\sqrt{\mu})}\log\big(k\big((L+\alpha)/\varepsilon+k\alpha\sqrt{\mu}/(M\sqrt{\varepsilon})\big)\big)\Big)\\
&\eqb O\Big(n{\sum}_{k=1}^{K_{\sf cons}}\big(\sqrt{(L+\alpha)/\mu}+\sqrt{k\alpha/M}({\varepsilon}/{\mu})^{1/4}\big)(\log k+\log\big((L+\alpha){\mu}/(M\varepsilon)\big)\big)\Big)\\
&\eqc O\Big(n\big(\sqrt{(L+\alpha)/\mu}K_{\sf cons}\big(\log K_{\sf cons}+ \log\big((L+\alpha){\mu}/(M\varepsilon)\big)\big)\\
&\quad\quad\quad\quad\quad\quad+\sqrt{\alpha/M}({\varepsilon}/{\mu})^{1/4}K_{\sf cons}^{3/2}(\log K_{\sf cons}+ \log\big((L+\alpha){\mu}/(M\varepsilon)\big)\big)\Big)\\
& = O\Big(n\big(M\sqrt{L+\alpha}/(\mu\sqrt{\varepsilon})+M\sqrt{\alpha}/(\mu\sqrt{\varepsilon})\big)\log\big((L+\alpha)/\varepsilon\big)\Big),
\end{array}
\end{align*}
where in (a) we use $\gamma_k=\Theta(\varepsilon/k)$, in (b) we use $\alpha/\sqrt{\varepsilon}=O((L+\alpha)/\varepsilon)$ 
and in (c) we use $\sum_{k=1}^K k^\nu\log k = \Theta(K^{\nu+1}\log K)$, for any $\nu\ge 0$. 
By noting that $\alpha\le L+\alpha$, we obtain~\eqref{eq:comp_CS_det}. 

 
\bibliographystyle{plain}
\bibliography{math_opt,mach_learn,stoc_ref}

\end{document}